\documentclass{amsart}
\usepackage{amssymb}
\usepackage{hyperref}

\newcommand{\C}{\mathbb{C}}
\newcommand{\Z}{\mathbb{Z}}
\newcommand{\bK}{\mathbb{K}}
\newcommand{\bO}{\mathbb{O}}

\newcommand{\GL}{\mathrm{GL}}

\newcommand{\Gr}{\mathcal{G}r}
\newcommand{\Fl}{\mathcal{F}l}
\newcommand{\bGr}{\mathbf{Gr}}

\newcommand{\gen}{{\boldsymbol{\eta}}}
\newcommand{\spc}{{\mathbf{s}}}
\newcommand{\uL}{\underline{L}}
\newcommand{\bt}{\mathbf{t}}
\newcommand{\bL}{\mathbf{L}}
\newcommand{\ubL}{\underline{\mathbf{L}}}

\newcommand{\cA}{\mathcal{A}}
\newcommand{\bX}{\mathbb{X}}
\newcommand{\fS}{\mathfrak{S}}
\newcommand{\Wext}{W_{\mathrm{ext}}}
\newcommand{\co}{{\mathrm{co}}}

\newcommand{\fancy}{{\mathrm{fancy}}}
\newcommand{\dom}{\mathsf{dom}}
\newcommand{\cox}{\mathsf{c}}
\newcommand{\rot}{\mathrm{rot}}
\newcommand{\sP}{\mathsf{P}}
\newcommand{\Fix}{\mathsf{Fix}}

\newcommand{\ttimes}{\mathbin{\widetilde{\mathord{\times}}}}
\DeclareMathOperator{\spn}{span}

\DeclareMathOperator{\cyc}{cyc}
\newcommand{\mult}{\mathsf{m}}
\newcommand{\bmult}{\boldsymbol{\mathsf{m}}}
\newcommand{\cV}{\mathcal{V}}
\newcommand{\bV}{\mathbf{V}}

\newtheorem{thm}{Theorem}[section]
\newtheorem{lem}[thm]{Lemma}
\newtheorem{prop}[thm]{Proposition}
\newtheorem{cor}[thm]{Corollary}
\theoremstyle{remark}
\newtheorem{rmk}[thm]{Remark}
\newtheorem{ex}[thm]{Example}

\numberwithin{equation}{section}

\title[Global Schubert varieties]{On global Schubert varieties for the general linear group}
\author{Pramod N. Achar}
\email{pramod.achar@math.lsu.edu}
\address{Department of Mathematics, Louisiana State University, Baton Rouge, LA \ 70803, USA}

\thanks{This work was partially supported by NSF Grant Nos.~DMS-1802241, DMS-2202012, and DMS-2231492.}

\author{Andrea Bourque}
\email{abourque@uchicago.edu}
\address{Department of Mathematics, University of Chicago, Chicago, IL \ 60637, USA}

\begin{document}

\begin{abstract}
We give a ``lattice-theoretic'' description of the global Schubert variety for $\GL_n$ associated to any dominant coweight. 
\end{abstract}

\maketitle

\section{Introduction}

Let $G$ be a complex reductive group.  One can associate to $G$ its \emph{affine flag variety} $\Fl_G$ and its \emph{affine Grassmannian} $\Gr_G$, both of which play important roles in geometric representation theory.  The orbits of $G(\C[[t]])$ on $\Gr_G$ are labelled by the set of dominant coweights $\bX^+$ for $G$.  For $\lambda \in \bX^+$, let $\Gr_{G,\lambda}$ be the corresponding orbit.  Its closure $\overline{\Gr_{G,\lambda}}$ is called a \emph{spherical Schubert variety}.

The \emph{global affine Grassmannian}, denoted $\bGr_G$, is a space that exhibits $\Fl_G$ as a degeneration of $\Gr_G$.  In more detail, $\bGr_G$ is equipped with a structure map $\rho: \bGr_G \to \C$, and we define its \emph{generic part} and \emph{special fiber} by
\[
\bGr_{G,\gen} = \rho^{-1}(\C^\times)
\qquad\text{and}\qquad
\bGr_{G,\spc} = \rho^{-1}(0),
\]
respectively.  A key fact about $\bGr_G$ is that we have identifications
\[
\bGr_{G,\gen} = \C^\times \times \Gr_G
\qquad\text{and}\qquad
\bGr_{G,\spc} = \Fl_G.
\]
One version of $\bGr_G$ was introduced in~\cite{gai:cce}; we will work a variant developed in~\cite{pz:lmsv,zhu:ccpr}.

Now let $\lambda \in \bX^+$.  The variety $\C^\times \times \Gr_{G,\lambda}$ can be identified with a locally closed subvariety of (the generic part of) $\bGr_G$.  Its closure in $\bGr_G$, denoted by $\overline{\bGr_{G,\lambda}}$, is called a \emph{global Schubert variety}.  

For the remainder of the paper, we focus on the case $G = \GL_n(\C)$, and we omit the subscript $G$ from $\Fl$, $\Gr$, and $\bGr$.  It is well known that $\Fl$ and $\Gr$ admit descriptions in terms of \emph{lattices}.  These descriptions will be reviewed in Section~\ref{sec:lattices}.  We will also review the lattice-theoretic description of the spherical Schubert varieties $\overline{\Gr_\lambda}$.

The goal of this paper is to give a lattice-theoretic description of the global Schubert varieties $\overline{\bGr_\lambda}$.  (A statement along these lines has been known to researchers working on Shimura varieties for some time: see the remarks after Theorem~\ref{thm:main}.)  To formulate the main result, we need a lattice-theoretic description of $\bGr$ for $\GL_n(\C)$.  This description (to be explained in Section~\ref{sec:global}) realizes $\bGr$ as a subset
\[
\bGr \subset \C \times \underbrace{\Gr \times \cdots \times \Gr}_{\text{$n$ copies}},
\]
where the structure map $\rho: \bGr \to \C$ is projection onto the first factor. For a dominant coweight $\lambda$, we define a closed subset $X_\lambda \subset \bGr$ by
\[
X_\lambda = \bGr \cap (\C \times \overline{\Gr_\lambda} \times \cdots \times \overline{\Gr_\lambda}).
\]
See Remark~\ref{rmk:global-exp} for a very concrete description of $X_\lambda$.  The main result of this paper is the following characterization of global Schubert varieties for $\GL_n(\C)$.

\begin{thm}\label{thm:main}
For any dominant coweight $\lambda$, we have $\overline{\bGr_\lambda} = X_\lambda$.
\end{thm}

Readers who are familiar with the theory of Shimura varieties will recognize this statement as asserting the ``topological flatness of (certain) local models of Shimura varieties.''  The ``local models'' in question are certain schemes over $\mathrm{Spec}\  \Z_p$, but when the same constructions are carried out with $\Z_p$ replaced by $\C[[t]]$, the resulting schemes are essentially the same as global Schubert varieties. The idea of describing local models of Shimura models in terms of lattices was first developed systematically by Rapoport--Zink~\cite{rz:pspdg}.  In this setting, for $\lambda$ minuscule, Theorem~\ref{thm:main} was proved by G\"ortz~\cite{goe:fcmsv}.  For general $\lambda$, one can deduce Theorem~\ref{thm:main} from a result of Haines--Ng\^o~\cite{hn:nclm} on nearby cycles for local models; see~\cite[\S7]{hai:isvbr} for a sketch of this argument.  A similar short proof of Theorem~\ref{thm:main} can be obtained by combining results of Kottwitz--Rapoport~\cite{kr:ma}, Haines--Ng\^o~\cite{hn:aasf}, and Zhu~\cite{zhu:ccpr}.  (We are grateful to T.~Haines for these pointers to the literature.)

The present paper aims to give a self-contained proof of Theorem~\ref{thm:main} using only basic topology and linear algebra.  (In particular, there is no mention of either Shimura varieties or nearby cycles later in the paper.)  We hope that this may make the statement more accessible to readers who are unfamiliar with the theory of local models.

Here is an outline of the proof.  Once the preliminary definitions are in place, it is relative easy to see that $\overline{\bGr_\lambda}$ is contained in $X_\lambda$, and that they agree on the generic part:
\[
\bGr_\gen \cap \overline{\bGr_\lambda} = \bGr_\gen \cap  X_\lambda = \C^\times \times \overline{\Gr_\lambda}.
\]
Thus, to prove Theorem~\ref{thm:main}, the main task to is to show that the special fiber of $X_\lambda$ is not too big.  We do this as follows:
\begin{enumerate}
\item When $\lambda$ is minuscule, we write down explicit formulas exhibiting each point of $X_\lambda$ as a limit of a $1$-parameter family in $\bGr_\gen \cap \overline{\bGr_\lambda}$.\label{it:intro1}
\item For the general case, we use the geometry of (global) \emph{convolution spaces} to obtain the result by induction on $\lambda$.\label{it:intro2}
\end{enumerate}
Step~\ref{it:intro1} is equivalent to~\cite[Proposition~4.16]{goe:fcmsv}, which is also proved by an explicit calculation (but not the same as ours).  It may be possible to treat general $\lambda$ by a similar method (using~\cite[Lemma~2]{goe:tflm}), but we have not pursued this approach here.


\subsection*{Outline of the paper}
Section~\ref{sec:lattices} contains background and notation related to lattices, $\Fl$, and $\Gr$. Section~\ref{sec:global} contains background related to $\bGr$ and global Schubert varieties.  The minuscule case of Theorem~\ref{thm:main} is proved in Section~\ref{sec:minuscule}.  Next, Section~\ref{sec:alcoves} establishes a technical combinatorial statement that is related to the geometry of convolution spaces.  Finally, in Section~\ref{sec:conv}, we review the construction of convolution spaces, and we complete the proof of Theorem~\ref{thm:main}.

\section{Background on lattices}
\label{sec:lattices}

\subsection{Group-theoretic preliminaries}

Let $\bK = \C((t))$ be the field of formal Laurent series over $\C$ with indeterminate $t$, and let $\bO = \C[[t]]$ be the ring of formal power series.  Fix an integer $n \ge 1$.  We will consider the groups $\GL_n(\C)$, $\GL_n(\bO)$, and $\GL_n(\bK)$.  For $1 \le i \le n$, let
\[
e_i = (0,\ldots,0,1,0, \ldots ,0) \in \Z^n
\]
be the vector whose $i$-th coordinate is $1$, with all other coordinates $0$.  In some contexts, we may also regard $e_i$ as a vector in $\C^n$ or $\bK^n$.

Let $T \subset \GL_n(\C)$ be the maximal torus consisting of diagonal matrices, and let $\bX$ be its coweight lattice.  Identify $\bX$ with $\Z^n$ as follows: if $\lambda = (\lambda_1, \ldots, \lambda_n) \in \Z^n$, the corresponding cocharacter $\C^\times \to T$ is 
\[
t \mapsto 
\left[
\begin{smallmatrix}
t^{\lambda_1} \\
& \ddots \\
&& t^{\lambda_n}
\end{smallmatrix}
\right].
\]
As a shorthand, we write $\bt^\lambda$ for the matrix above.  Then $\bt^\lambda$ can be regarded as an element of $\GL_n(\bK)$.  We define the \emph{size} of a coweight $\lambda$ to be
\[
|\lambda| = \lambda_1 + \cdots + \lambda_n.
\]

For $0 \le i \le n$, let
\[
\varpi_i = e_1 + \ldots + e_i = (\underbrace{1,\ldots,1}_{\text{$i$ terms}}, \underbrace{0,\ldots,0}_{\text{$n-i$ terms}}) \in \bX.
\]

Let $B^- \subset \GL_n(\C)$ be the subgroup consisting of lower-triangular matrices.  We define \emph{dominant} coweights with respect to the opposite Borel subgroup to $B^-$: thus, a coweight $\lambda = (\lambda_1, \ldots, \lambda_n)$ is dominant if
\[
\lambda_1 \ge \cdots \ge \lambda_n.
\]
The set of dominant coweights is denoted by $\bX^+ \subset \bX$.  Let $\preceq$ be the usual partial order on $\bX$: we set $\mu \preceq \lambda$ if $\lambda - \mu$ is a nonnegative linear combination of the simple coroots $(0,\ldots, 0,1,-1,0,\ldots,0)$.  In concrete terms, for $\mu = (\mu_1, \ldots, \mu_n)$ and $\lambda = (\lambda_1, \ldots, \lambda_n)$, we have
\[
\mu \preceq \lambda
\qquad\text{if}\qquad
\text{$|\mu| = |\lambda|$\quad  and\quad $\sum_{i=1}^k \mu_i \le \sum_{i=1}^k \lambda_i$\quad for $1 \le k \le n$}.
\]

Let $W = \fS_n$ be the Weyl group of $\GL_n(\C)$, identified with the symmetric group on $n$ letters.  This group acts on $\bX$ in the obvious way.  For each $\lambda \in \bX$, there is a unique dominant coweight in its $\fS_n$-orbit.  We denote it by
\[
\dom(\lambda) \in \fS_n \cdot \lambda \cap \bX^+.
\]
Let
\[
\Wext= \fS_n \ltimes \bX
\]
be the extended affine Weyl group.  

For $\sigma \in \fS_n$, let $\dot\sigma \in \GL_n(\C)$ be the corresponding permutation matrix.  Now suppose $w = (\sigma,\lambda) \in \Wext$.  We set
\[
\dot w = \dot \sigma \bt^\lambda \in \GL_n(\bK).
\]

There is a homomorphism $\mathrm{ev}_0: \GL_n(\bO) \to \GL_n(\C)$ given by $t \mapsto 0$.  We set
\[
I = (\mathrm{ev}_0)^{-1}(B^-) \subset \GL_n(\bO).
\]
This is an \emph{Iwahori subgroup} of $\GL_n(\bO)$.

\subsection{Lattices}
\label{ss:lattices}

A \emph{lattice} in the vector space $\bK^n$ is defined to be a free $\bO$-submodule $L \subset \bK^n$ of rank $n$.  In particular, $\bO^n$ is called the \emph{standard lattice}.  For $\lambda \in \bX$, let
\[
\bL^\lambda = \bt^\lambda \cdot \bO^n = \text{the lattice spanned by the columns of $\bt^\lambda$}.
\]
For example, we have
\[
\bL^{\varpi_i} = (t\bO)^i \times \bO^{n-i}.
\]

Given a lattice $L$, let $f_1, \ldots, f_n \in \bK^n$ be an $\bO$-basis for $L$.  Then one can form the matrix whose columns are $f_1, \ldots, f_n$.  Then there is an integer $\nu(L)$ such that
\[
\det \begin{bmatrix} f_1 & \cdots & f_n \end{bmatrix} = t^{\nu(L)} \cdot(\text{a unit in $\bO$}).
\]
Equivalently, $\nu(L)$ is the smallest integer such that $t^{\nu(L)}$ occurs with nonzero coefficient in $\det \begin{bmatrix} f_1 & \cdots & f_n \end{bmatrix}$.  This integer is independent of the choice of basis.  It is called the \emph{valuation} of $L$.  We clearly have
\[
\nu(\bL^\lambda) = |\lambda|.
\]

If $L' \subset L$ are two lattices, then $L/L'$ is a finite-dimensional $\C$-vector space with dimension $\nu(L') - \nu(L)$.  A \emph{lattice chain} is a sequence of lattices
\[
\uL = (L_1 \supset L_2 \supset \cdots \supset L_n \supset tL_1)
\]
such that $\dim L_i/L_{i+1} = 1$ for $1 \le i \le n-1$.  (This condition forces $\dim L_n/tL_1 = 1$ as well.)  The \emph{standard lattice chain}, denoted by $\ubL^e$, is the lattice chain
\[
\ubL^e = (\bL^0 \supset \bL^{\varpi_1} \supset \bL^{\varpi_2} \supset \cdots \supset \bL^{\varpi_{n-1}} \supset t\bL^0 = \bL^{\varpi_n}).
\]
For $w = (\sigma,\lambda) \in \Wext$, we set
\[
\ubL^w = \dot w \cdot \ubL^e = (\bL^{\sigma(\lambda)} \supset \bL^{\sigma(\lambda+\varpi_1)} \supset \cdots \supset \bL^{\sigma(\lambda + \varpi_{n-1})} \supset t\bL^{\sigma(\lambda)}).
\]

\subsection{The affine Grassmannian and the affine flag variety}

The \emph{affine Grassmannian} $\Gr$ and the \emph{affine flag variety} $\Fl$ are defined by
\[
\begin{aligned}
\Gr &= \{ \text{lattices in $\bK^n$} \} &&= \GL_n(\bK)/\GL_n(\bO), \\
\Fl &= \{ \text{lattice chains in $\bK^n$}\} &&= \GL_n(\bK)/I.
\end{aligned}
\]
The latter identifications above come from the following observation: $\GL_n(\bK)$ acts transitively on the set of all lattices and the set of all lattice chains, and the stabilizer of the standard lattice (resp.~standard lattice chain) is $\GL_n(\bO)$ (resp.~$I$).

Both $\Gr$ and $\Fl$ can be made into ind-projective ind-varieties: that is, each can be written as a union $\bigcup_{i \ge 1} X_i$, where each $X_i$ is a projective complex variety, and the inclusion maps $X_i \to X_{i+1}$ are closed immersions.  See, for instance,~\cite{kum:kmg}.  

It is well known that the set of $\GL_n(\bO)$-orbits on $\Gr$ is in bijection with $\bX^+$.  Explicitly, for $\lambda \in \bX^+$, the corresponding orbit is given by
\[
\Gr_\lambda = \GL_n(\bO) \cdot \bL^\lambda \subset \Gr.
\]
On the other hand, the $I$-orbits on $\Gr$ are in bijection with all of $\bX$.  We have
\begin{equation}\label{eqn:gr-i-orbit}
\Gr_\lambda = \bigcup_{\lambda' \in \fS_n \cdot \lambda} I \cdot \bL^{\lambda'}.
\end{equation}
The closure $\overline{\Gr_\lambda}$ is called a \emph{spherical Schubert variety}.  For $\mu, \lambda \in \bX^+$, we have
\[
\Gr_\mu \subset \overline{\Gr_\lambda}
\qquad\text{if and only if}\qquad
\mu \preceq \lambda.
\]
Combining this with~\eqref{eqn:gr-i-orbit}, we find that
\begin{equation}\label{eqn:gr-i-closure}
I \cdot \bL^\mu \subset \overline{\Gr_\lambda} 
\qquad\text{if and only if}\qquad
\dom(\mu) \preceq \lambda.
\end{equation}

Similarly, the set of $I$-orbits on $\Fl$ is in bijection with $\Wext$: we have
\[
\Fl = \bigcup_{w \in \Wext} I \cdot \ubL^w.
\]

\begin{rmk}\label{rmk:spherical}
For a dominant coweight $\lambda \in \bX^+$, one can show that
\[
\Gr_\lambda = \left\{ L \in \Gr \,\Big|\,
\begin{array}{c}
\text{$\nu(L) = |\lambda|$, $t^{\lambda_1}\bO^n \subset L \subset t^{\lambda_n}\bO^n$, and for $\lambda_1 \ge i \ge \lambda_n$,} \\
\text{$\dim (L \cap t^i\bO^n)/t^{\lambda_1+1}\bO^n = \sum_{j = i}^{\lambda_1} |\{ k \mid \lambda_k \le i \}|$} 
\end{array}
\right\},
\]
and
\[
\overline{\Gr_\lambda} = \left\{ L \in \Gr \,\Big|\,
\begin{array}{c}
\text{$\nu(L) = |\lambda|$, $t^{\lambda_1}\bO^n \subset L \subset t^{\lambda_n}\bO^n$, and for $\lambda_1 \ge i \ge \lambda_n$,} \\
\text{$\dim (L \cap t^i\bO^n)/t^{\lambda_1+1}\bO^n \ge \sum_{j = i}^{\lambda_1} |\{ k \mid \lambda_k \le i \}|$} 
\end{array}
\right\}.
\]
See, for instance,~\cite[Exercise~9.1.5]{a:psart} (which has a minor typo in it).
\end{rmk}

\begin{ex}\label{ex:locfund}
For $\lambda = \varpi_k$ with $1 \le k \le n-1$, Remark~\ref{rmk:spherical} reduces to
\[
\Gr_{\varpi_k} = \overline{\varpi_k} = \{ L \in \Gr \mid \text{$t\bO^n \subset L \subset \bO^n$ and $\dim L/t\bO^n = n - k$} \}.
\]
Identify $\bO^n/t\bO^n \cong \C^n$, and let $G(n,n-k)$ be the usual Grassmannian of $(n-k)$-dimensional subspaces in $\C^n$.  Then we have an isomorphism
\[
\Gr_{\varpi_k} \cong G(n,n-k)
\]
given by $L \mapsto L/t\bO^n \in G(n,n-k)$.  
\end{ex}

\subsection{Kottwitz--Rapoport alcoves}

Define a new partial order $\le_\co$ (here, ``$\co$'' stands for ``coordinatewise'') on $\bX$ as follows: for coweights $\lambda = (\lambda_1,\ldots,\lambda_n)$ and $\mu = (\mu_1, \ldots, \mu_n)$, we set
\[
\lambda \le_\co \mu
\qquad\text{if}\qquad
\text{$\lambda_i \le \mu_i$ for $1 \le i \le n$.}
\]

Now let $x = (x^{(1)}, \ldots, x^{(n)}) \in \bX^n$ be an $n$-tuple of coweights.  Following Kott\-witz--Rapoport~\cite{kr:ma}, we call $x$ an \emph{alcove} if the following conditions hold:
\begin{gather*}
x^{(1)} \le_\co x^{(2)} \le_\co \cdots \le_\co x^{(n)} \le_\co x^{(1)} + \varpi_n, \\
|x^{(2)}| = |x^{(1)}|+1,
\quad
|x^{(3)}| = |x^{(2)}|+1,
\quad
\ldots
\quad
|x^{(n)}| = |x^{(n-1)}|+1.
\end{gather*}
When working with alcoves, it is sometimes convenient to adopt the convention that $x^{(n+1)} = x^{(1)} + \varpi_n$.

The definition implies that $x^{(i+1)}$ differs from $x^{(i)}$ in exactly one coordinate.  For an alcove $x$, let $\sP(x) \in \fS_n$ be the permutation characterized by the property that
\[
x^{(i+1)}_j =
\begin{cases}
x^{(i)}_j + 1 & \text{if $j = \sP(x)(i)$,} \\
x^{(i)}_j & \text{otherwise.}
\end{cases}
\]
Alternatively, $x^{(i+1)} = x^{(i)} + e_{\sP(x)(i)} = x^{(i)} + \sP(x)e_i$.  By induction on $i$, we find
\[
x^{(i)} = x^{(1)} + \sP(x)\varpi_{i-1}.
\]
The \emph{base alcove} is the alcove
\[
\omega = (\varpi_0, \varpi_1, \ldots, \varpi_{n-1}).
\]
It has the property that $\sP(\omega)$ is the identity permutation.

Let $\cA \subset \bX^n$ be the set of alcoves.  There is an action of $\Wext$ on $\cA$, given by
\[
(\sigma,\lambda) \cdot (x^{(1)}, \ldots, x^{(n)}) = (\sigma(\lambda + x^{(1)}), \ldots, \sigma(\lambda + x^{(n)})).
\]
As explained in~\cite[\S3.2]{kr:ma}, this action is simply transitive, so there is a bijection
\[
\Wext \overset{\sim}{\leftrightarrow} \cA
\qquad\text{given by}\qquad
w \mapsto w \cdot \omega.
\]
The inverse map is given by $x \mapsto( \sP(x), \sP(x)^{-1}\cdot x^{(1)})$.

This bijection gives us an alternative way to label the $I$-orbits on $\Fl$. In particular, if $x = (x^{(1)}, \ldots, x^{(n)})$ is an alcove, let
\[
\ubL^x = (\bL^{x^{(1)}}\supset \bL^{x^{(2)}} \supset \cdots \supset \bL^{x^{(n)}} \supset t\bL^{x^{(1)}}).
\]
If $x = w \cdot \omega$ for $w \in \Wext$, then $\ubL^x = \ubL^w$.

Let $\lambda \in \bX^+$ be a dominant coweight.  An alcove $x = (x^{(1)}, \ldots, x^{(n)})$ is said to be \emph{$\lambda$-permissible} if we have
\[
\dom(x^{(i)} - \varpi_{i-1}) \preceq \lambda
\qquad\text{for $1 \le i \le n$.}
\]
For $x, y \in \cA$, we say that $x$ is in \emph{relative position $\lambda$} with respect to $y$ if we have
\[
\dom(x^{(i)} - y^{(i)}) = \lambda
\qquad\text{for $1 \le i \le n$.}
\] 

\section{The global affine Grassmannian}
\label{sec:global}

For $1 \le i \le n$ and $y \in \C$, let $\theta_i(y): \bK^n \to \bK^n$ be the linear map defined by
\[
\theta_i(y)(v_1, \ldots, v_n) = ((t-y)v_1, (t-y)v_2, \ldots, (t-y)v_i, v_{i+1}, \ldots, v_n).
\]
In other words, $\theta_i(y)$ is defined by substituting $t-y$ for $t$ in the matrix $\bt^{\varpi_i}$.  Since $t - y \ne 0$, the map $\theta_i(y)$ is always an isomorphism of $\bK$-vector spaces.  It restricts to a map $\theta_i(y): \bO^n \to \bO^n$.  The latter map is an isomorphism if and only if $y \ne 0$ (equivalently, if and only if $t -y$ is invertible in $\bO$). 

We define a \emph{lattice family} $F$ to be a tuple
\[
F = (y, L_1, \ldots, L_n)
\]
where $y \in \C$ and $L_1, \ldots, L_n \in \Gr$, and where the following two conditions hold:
\begin{gather}
\nu(L_1) = \cdots = \nu(L_n), \label{eqn:bgr-val} \\
L_1 \supset \theta_1(y)(L_2) \supset \theta_2(y)(L_3) \supset \cdots \supset \theta_{n-1}(y)(L_n) \supset (t-y)L_1. \label{eqn:bgr-chain}
\end{gather}

We define the \emph{global affine Grassmannian} to be the set
\begin{equation}\label{eqn:bgr-defn}
\bGr = \{ \text{lattice families} \} \subset \C \times \underbrace{\Gr \times \cdots \times \Gr}_{\text{$n$ copies}}.
\end{equation}
This set inherits the structure of an ind-variety from  $\C \times \Gr \times \cdots \times \Gr$.  There is an obvious map $\rho: \bGr \to \C$ given by $\rho(y, L_1, \ldots, L_n) = y$.  Let
\[
\bGr_\gen = \rho^{-1}(\C \smallsetminus \{0\})
\qquad\text{and}\qquad
\bGr_\spc = \rho^{-1}(0).
\]
We call $\bGr_\gen$ the \emph{generic part} of $\bGr$, and we call $\bGr_\spc$ the \emph{special fiber}.

\begin{lem}\label{lem:bgr-gen-spc}
There are isomorphisms of ind-varieties
\[
\bGr_\gen \cong \C^\times \times \Gr,
\qquad
\bGr_\spc \cong \Fl.
\]
\end{lem}
\begin{proof}
When $y \ne 0$, the maps $\theta_i(y)$ all lie in $\GL_n(\bO)$; in particular, the action of $\theta_i(y)$ on $\Gr$ preserves valuation.  Since all lattices in the sequence~\eqref{eqn:bgr-chain} have the same valuation, the containments are equalities.  Thus, the map $\bGr_\gen \to \C^\times \times \Gr$ given by $(y, L_1, \ldots, L_n) \mapsto (y, L_1)$ is an isomorphism, with inverse given by
\[
(y,L) \mapsto (y, L, \theta_1(y)^{-1}(L), \theta_2(y)^{-1}(L), \ldots, \theta_{n-1}(y)^{-1}(L)).
\]

When $y = 0$, the action of $\theta_i(y)$ on a lattice increases its valuation by $1$, so the sequence in~\eqref{eqn:bgr-chain} is a lattice chain in the sense of~\ref{ss:lattices}.  Thus, the isomorphism $\bGr_\spc \to \Fl$ is given by
\[
(0,L_1, \ldots, L_n) \mapsto (L_1 \supset \theta_1(0)L_2 \supset \cdots \supset \theta_{n-1}(0)(L_n) \supset tL_1).\qedhere
\]
\end{proof}

\begin{rmk}\label{rmk:fancy}
For a complex reductive group $G$, a ``fancy'' definition of the associated global affine Grassmannian $\bGr^\fancy_G$ can be found in~\cite[\S6.2.3]{pz:lmsv} or \cite[\S3.1]{zhu:ccpr}.  In this definition, $\bGr^\fancy_G$ is the ind-scheme that classifies principal bundles for a certain ``parahoric-style'' group scheme related to $G$.  An expository account can be found in~\cite[\S2.2.3.2]{ar:csafv}, and the case where $G = \GL_n$ is worked out in~\cite[Example~2.2.16]{ar:csafv}.  That example shows that $\bGr^\fancy_{\GL_n}$ agrees with~\eqref{eqn:bgr-defn} at the level of $\C$-points.

However, they don't quite agree algebro-geometrically: it turns out that the ind-scheme $\bGr^\fancy_{\GL_n}$ is not reduced, and that the variety defined in~\eqref{eqn:bgr-defn} is its reduced subscheme.  Of course, this distinction is irrelevant for the study of global Schubert varieties, which are reduced by definition.
\end{rmk}

For $\lambda \in \bX^+$, consider the orbit $\Gr_\lambda \subset \Gr$.  Via Lemma~\ref{lem:bgr-gen-spc}, we may regard $\C^\times \times \Gr_\lambda$ as a locally closed subset of $\bGr_\gen$, and we can take its closure $\overline{\C^\times \times \Gr_\lambda} \subset \bGr$.  We call this the \emph{global Schubert variety} labelled by $\lambda$, and we denote it by
\[
\overline{\bGr_\lambda} = \overline{\C^\times \times \Gr_\lambda}.
\]
(Note that the notation ``$\bGr_\lambda$'' is not defined.)  The main result of this paper (Theorem~\ref{thm:main}) states that $\overline{\bGr_\lambda}$ coincides with the variety
\[
X_\lambda = \{ (y, L_1, \ldots, L_n) \in \bGr \mid L_1, \ldots, L_n \in \overline{\Gr_\lambda} \}.
\]
It is clear that $X_\lambda$ is a closed subset of $\bGr$, and that
\[
X_\lambda \cap \bGr_\gen = \C^\times \times \overline{\bGr_\lambda}.
\]
It follows immediately that $X_\lambda \supset \overline{\bGr_\lambda}$. Thus, Theorem~\ref{thm:main} comes down to proving the opposite containment.  Moreover, since we have already shown that this is an equality over $\C^\times$, what remains is to show that
\[
X_\lambda \cap \bGr_\spc \subset \overline{\bGr_\lambda} \cap \bGr_\spc.
\]
Via Lemma~\ref{lem:bgr-gen-spc}, we identify $X_\lambda \cap \bGr_\spc$ with a subset of $\Fl$. 

\begin{lem}\label{lem:x-adm}
Let $\lambda$ be a dominant coweight.
\begin{enumerate}
\item The varieties $\overline{\bGr_\lambda} \cap \bGr_\spc, X_\lambda \cap \bGr_\spc \subset \Fl$ are both stable under $I$.\label{it:xa-orbit}
\item For $x \in \cA$, we have $I \cdot \ubL^x \subset X_\lambda \cap \bGr_\spc$ if and only if $x$ is $\lambda$-permissible.\label{it:xa-perm}
\end{enumerate}
\end{lem}
\begin{proof}
\eqref{it:xa-orbit}~Define an action of $I$ on $\C \times \Gr \times \cdots \times \Gr$ as follows: for $g \in I$,
\[
g \cdot (y, L_1, \ldots, L_n) = (y, g L_1, \theta_1(y)^{-1}g\theta_1(y)L_2, \ldots, \theta_{n-1}(y)^{-1}g\theta_{n-1}(y)L_n).
\]
This ``global'' action preserves conditions~\eqref{eqn:bgr-val} and~\eqref{eqn:bgr-chain}, so it restricts to an action on $\bGr$.  When $y = 0$, the global action coincides with the natural action on $I$ on $\Fl$ via the isomorphism of Lemma~\ref{lem:bgr-gen-spc}.   

A short calculation shows that for any $g \in I$ and any $y \in \C$, the matrix $\theta_{i-1}(y)^{-1}g \theta_{i-1}(y)$ lies in $\GL_n(\bO)$.  It is immediate that $X_\lambda$ and $\C^\times \times \Gr_\lambda$ are stable under the global $I$-action; the latter case implies that $\overline{\bGr_\lambda}$ is as well.    

\eqref{it:xa-perm}~Let $x = (x^{(1)}, \ldots, x^{(n)})$ be an alcove.  The lattice chain $\ubL^x = (\bL^{x^{(1)}} \supset \bL^{x^{(2)}} \supset \cdots \supset \bL^{x^{(n)}})$ lies in $X_\lambda \cap \bGr_\spc$ if and only if
\[
\theta_{i-1}(0)^{-1}\bL^{x^{(i)}} \in \overline{\bGr_\lambda}
\]
for $1 \le i \le n$.  (When $i =1$, $\theta_0(0)$ is the identity map.)  From the definition, we see that $\theta_{i-1}(0)^{-1} \bL^{x^{(i)}} = \bL^{x^{(i)} - \varpi_{i-1}}$.  By~\eqref{eqn:gr-i-closure}, we conclude that $I \cdot \ubL^x \subset X_\lambda \cap \bGr_\spc$ if and only if $\dom(x^{(i)} - \varpi_{i-1}) \preceq \lambda$,
i.e., if and only if $x$ is $\lambda$-permissible.
\end{proof}

In view of Lemma~\ref{lem:x-adm}, Theorem~\ref{thm:main} comes down to the following claim:
\begin{equation}\label{eqn:main-claim}
\text{For each $\lambda$-permissible alcove $x$, we have $\ubL^x \in \overline{\bGr_\lambda} \cap \bGr_\spc$.}
\end{equation}
The following lemma describes one easy case where this condition holds.

\begin{lem}\label{lem:adm-constant}
If $\lambda = (a,a,\ldots,a)$ is a constant coweight, then $X_\lambda = \overline{\bGr_\lambda}$.
\end{lem}
\begin{proof}
In this case, $\lambda$ is minimal with respect to $\preceq$, and $\Gr_\lambda = \overline{\Gr_\lambda}$ is a single point.  It follows that $X_\lambda$ and $\overline{\bGr_\lambda}$ are both isomorphic to $\C$.
\end{proof}

\begin{rmk}\label{rmk:global-exp}
Using Remark~\ref{rmk:spherical}, we can write out the definition of $X_\lambda$ as
\[
\left\{ (y, L_1, \ldots, L_n) \,\left|
\begin{array}{c}
\text{$y \in \C$, and $L_1, \ldots, L_n$ are lattices in $\bK^n$ such that} \\
\text{$L_1 \supset \theta_1(y)(L_2) \supset \cdots \supset \theta_{n-1}(y)(L_n) \supset (t-y)L_n$, and such} \\
\text{that for $1 \le i \le n$ and $\lambda_1 \ge j \ge \lambda_n$, we have}  \\
\text{ $t^{\lambda_1}\bO^n\subset L_i \subset t^{\lambda_n}\bO^n$, $\nu(L_i) = |\lambda|$, and} \\
\text{$\dim (L_i \cap t^j\bO^n)/t^{\lambda_1+1}\bO^n \ge \sum_{k = j}^{\lambda_1} |\{ h \mid \lambda_h \le k \}|$}
\end{array}
\right.\right\}.
\]
Theorem~\ref{thm:main} says that this is also a description of $\overline{\bGr_\lambda}$.
\end{rmk}

\begin{rmk}\label{rmk:permissible}
The notion of \emph{$\lambda$-permissibility} was introduced in~\cite{kr:ma}, along with a related notion called \emph{$\lambda$-admissibility}, which is defined in terms of the Bruhat order on $\Wext$.  (Neither $\lambda$-admissibility nor the Bruhat order play any role in the present paper.)  In~\cite[Theorem~1]{hn:aasf}, Haines--Ng\^o showed that $\lambda$-permissibility and $\lambda$-admissibility coincide for $G = \GL_n$ (although they do not coincide in general).  In~\cite[Theorem~3.8]{zhu:ccpr}, Zhu showed (for any $G$) that the special fiber $\overline{\bGr_\lambda} \cap \bGr_\spc$ is the union of $I$-orbits labelled by $\lambda$-admissible alcoves.  In view of Lemma~\ref{lem:x-adm}, we see that among the three statements
\[
\text{[this paper, Theorem~\ref{thm:main}]}
\qquad
\text{\cite[Theorem~1]{hn:aasf}}
\qquad
\text{\cite[Theorem~3.8 for $G = \GL_n$]{zhu:ccpr}}
\]
any two imply the third.
\end{rmk}

\section{The case of fundamental coweights}
\label{sec:minuscule}

Fix an integer $r$ with $1 \le r \le n-1$.  The goal of this section is prove~\eqref{eqn:main-claim}, and hence Theorem~\ref{thm:main}, for the case of a fundamental coweight $\lambda = \varpi_{n-r}$.  Recall from Example~\ref{ex:locfund} that $\Gr_{\varpi_{n-r}} \cong G(n,r)$.  Let us describe $X_{\varpi_{n-r}}$ in these terms.  For $1 \le i \le n$ and $y \in \C$, let
\[
\bar\theta_i(y): \C^n \to \C^n
\]
be the map induced by $\theta_i(y): \bO^n \to \bO^n$ under the identification $\bO^n/t\bO^n \cong \C^n$.  Explicitly, $\bar\theta_i(y)$ is given on basis vectors by
\begin{equation}\label{eqn:bartheta-defn}
\bar\theta_i(y)e_j =
\begin{cases}
-ye_j & \text{if $j \le i$,} \\
e_j & \text{if $j > i$.}
\end{cases}
\end{equation}
We have
\begin{multline}\label{eqn:xfund-defn}
X_{\varpi_{n-r}} =
\{ (y, H_1, \ldots, H_n) \in \C \times G(n,r) \times \cdots \times G(n,r) \mid{} \\
H_1 \supset \bar\theta_1(y)(H_2) \supset \cdots \supset \bar\theta_{n-1}(y)(H_n) \supset y H_1 \}.
\end{multline}
We will prove the result by giving an explicit formula for a $1$-parameter family in $X_{\varpi_{n-r}} \cap \bGr_\gen = \C^\times \times G(n,r)$ whose limit at $y = 0$ is a given point of $X_{\varpi_{n-r}} \cap \bGr_\spc$.

\subsection{Notation related to alcoves}
\label{ss:fundalcove}

Let $x = (x^{(1)}, \ldots, x^{(n)})$ be a $\varpi_{n-r}$-permissible alcove.  In particular, in the coweight $x^{(1)}$, there are $n-r$ coordinates equal to $1$, and $r$ coordinates equal to $0$. Let $J^{(1)} = (j^{(1)}_1, j^{(1)}_2, \ldots, j^{(1)}_r)$ be the set of indices (in some order) of the $0$'s in $x^{(1)}$, i.e.,
\[
x^{(1)} = \sum_{\substack{1 \le j \le n\\ j \notin J^{(1)}}} e_j.
\]
For $1 \le i \le n$, define a sequence $J^{(i+1)} = (j^{(i+1)}_1, \ldots, j^{(i+1)}_r)$ by the inductive rule
\begin{equation}\label{eqn:jq-defn}
j^{(i+1)}_q = 
\begin{cases}
i & \text{if $j^{(i)}_q = \sP(x)(i)$,} \\
j^{(i)}_q & \text{otherwise.}
\end{cases}
\end{equation}
Using the fact that $(x^{(i+1)} - \varpi_i) = (x^{(i)} - \varpi_{i-1}) + e_{\sP(x)(i)} - e_i$, we see that
\begin{equation}\label{eqn:alcove-J}
x^{(i)} - \varpi_{i-1} = \sum_{\substack{1 \le j \le n\\ j \notin J^{(i)}}} e_j.
\end{equation}
In particular, since $x^{(n+1)} - \varpi_n = x^{(1)}$, we see that $J^{(1)}$ and $J^{(n+1)}$ are equal as (unordered) sets.  As ordered sets, they are permutations of one another.  Let $\sigma \in \fS_r$ be the permutation such that
\[
j^{(1)}_q = j^{(n+1)}_{\sigma(q)} \qquad\text{for $1 \le q \le r$.}
\]

It will be useful to have an alternative notation for the sequence of integers
\[
j^{(1)}_q, j^{(2)}_q, \ldots, j^{(n+1)}_q.
\]
Let $f_q$ be $1 +{}$(the number of ``changes'' in this sequence), i.e.,
\[
f_q = 1 + |\{i \mid \text{$2 \le i \le n+1$ and $j^{(i)}_q \ne j^{(i-1)}_q$} \}|.
\]
We record the indices where changes occur as follows: let
\[
1 = c(q,1) < c(q,2) < \cdots < c(q,f_q) < c(q,f_q+1) = n+2
\]
be the sequence of integers characterized by
\[
j_q^{(c(q,f))} = j_q^{(c(q,f)+1)} = \cdots = j_q^{(c(q,f+1)-1)} \ne j_q^{(c(q,f+1))}
\qquad\text{for $1 \le f \le f_q$.}
\]
For $1 \le i \le n+1$, let $\phi(q,i) = \max \{f \mid \text{$1 \le f \le f_q$ and $c(q,f) \le i$}\}$.  Thus,
\[
c(q,\phi(q,i)) \le i < c(q,\phi(q,i)+1).
\] 
Finally, for $1 \le f \le f_q$, let
\[
a(q,f) = j_q^{(c(q,f))} = j_q^{(c(q,f)+1)} = \cdots = j_q^{(c(q,f+1)-1)}.
\]
In particular, $a(q,f_q) = j_q^{(n+1)}$, so
\begin{equation}\label{eqn:aq-sigma}
a(q,1) = a(\sigma(q), f_{\sigma(q)}).
\end{equation}

Suppose $1 \le q \le r$, and consider the finite set
\[
\cyc(q) = \{q, \sigma(q), \sigma^2(q), \ldots \} \subset \{1,\ldots, r\}.
\]
Let $M(q) \in \cyc(q)$ be the element whose corresponding term in $J^{(1)}$ is maximal:
\[
j^{(1)}_{M(q)} \ge j^{(1)}_s \qquad\text{for all $s \in \cyc(q)$.}
\]
Equivalently, $a(M(q),1) \ge a(s,1)$ for all $s \in \cyc(q)$.  We also set
\begin{gather*}
m(q) = \sigma(M(q)), \quad C(q) = |\cyc(q)|,\quad
\delta(q) = \min \{ h \ge 0 \mid \sigma^h(q) = M(q) \}.
\end{gather*}
The following observations are immediate from the definitions:
\begin{gather*}
0 = \delta(M(q)) \le \delta(q) \le \delta (m(q)) = C(q) - 1 < r\quad\text{for all $q$,} \\
\text{If $q \ne M(q)$, then }
\delta(\sigma(q)) = \delta(q) - 1, \\
\text{If $f_q = 1$, then $\phi(q,i)=1$ for all $i$,} \\
\text{If $f_q>1$, then $f_s>1$ for all $s\in\cyc(q)$.}
\end{gather*}

\begin{lem}\label{lem:aq-basic}
Let $1 \le q \le r$ and and $1 \le i \le n+1$.
\begin{enumerate}
\item If $2 \le f \le f_q$, then $a(q,f) = c(q,f) - 1$.\label{it:aqcq}
\item Suppose $2 \le f \le f_q$.  We have $a(q,f) < i$ if and only if $f \le \phi(q,i)$.\label{it:aqf}
\item We have $a(q,1) > a(q,2)$ and $a(q,2) < a(q,3) < \cdots < a(q,f_q)$.\label{it:aqorder}
\item For $1 \le f \le f_q$, we have $a(q,f) \le a(M(q),1)$.\label{it:aqM}
\item If $i > a(M(q),1)$, then $\phi(q,i) = f_q$.\label{it:Mphi}
\item If $f_q=n+1$, then $\sigma(q)=q$.\label{it:fqnr}
\end{enumerate}
\end{lem}
\begin{proof}
\eqref{it:aqcq}~Since $f \ge 2$, we have $j^{(c(q,f))}_q \ne j^{(c(q,f)-1)}_q$. From the definition~\eqref{eqn:jq-defn}, we have $j^{(c(q,f))}_q = c(q,f) - 1$.

\eqref{it:aqf}~In view of part~\eqref{it:aqcq}, we have $a(q,f) < i$ if and only if $c(q,f) \le i$, and the latter holds if and only if $f \le \phi(q,i)$.

\eqref{it:aqorder}~If $f_q=1$, then this is vacuous, as $a(q,f)$ is undefined for $f>f_q=1$. Otherwise, let $j = a(q,1)$.  Then $j \in J^{(1)}$ and $j \in J^{(n+1)}$ but $j \notin J^{(c(q,2))}$.  Let $c$ be the smallest integer with $c(q,2) < c \le n+1$ such that $j \in J^{(c)}$.  From~\eqref{eqn:jq-defn}, we must have $c = j+1$.  We deduce that $c(q,2) < a(q,1) + 1$, and then by part~\eqref{it:aqcq}, $a(q,2) < a(q,1)$.  The remaining inequalities in part~\eqref{it:aqorder} follow  from part~\eqref{it:aqcq}.

\eqref{it:aqM}~If $f = 1$, this holds by definition.  If $f \ge 2$, then $a(q,f) \le a(q,f_q) = a(\sigma^{-1}(q),1) \le a(M(\sigma^{-1}(q)),1) = a(M(q),1)$.

\eqref{it:Mphi}~If $f_q=1$, then $\phi(q,1)=1=f_q$ automatically. Otherwise, this follows from parts~\eqref{it:aqf} and~\eqref{it:aqM}.

\eqref{it:fqnr}~For $1\leq i \leq n$, we must have $j_q^{(i)} = \sP(x)(i) \neq i = j_q^{(i+1)}$. In particular, $i=\sP(x)(i+1)$ for $1\leq i \leq n-1$. It follows that $j_q^{(1)} = \sP(x)(1) = n = j^{(n+1)}_q$. Thus $\sigma(q)=q$. 
\end{proof}

\subsection{Families of vectors}
\label{ss:families}

In this subsection, for $1 \le q \le r$, we will define vectors $v_q(u), v_q'(u), v''_q(u)$ whose entries are polynomials in an indeterminate $u$.  To reduce clutter, throughout this subsection, we write
\begin{equation}\label{eqn:q-fix}
m = m(q),\qquad M = M(q), \qquad \delta = \delta(q), \qquad C = C(q).
\end{equation}

For $0 \le h \le \delta$, we introduce the notation
\[
\beta(h) = n^{r-h}
\qquad\text{and}\qquad
B(h) = \beta(0) + \beta(1) + \cdots + \beta(h-1),
\]
where we take $B(0)=0$.  Let
\[
v_q(u) = e_{a(q,1)} + \sum_{f=2}^{f_q} u^{\beta(\delta)(f_q +1 -f)} e_{a(q,f)}.
\]
If $f_q > 1$, we can rewrite this using~\eqref{eqn:aq-sigma} as
\begin{equation}
v_q(u) =
e_{a(q,1)} + u^{\beta(\delta)}e_{a(\sigma^{-1}(q),1)} + \sum_{f = 2}^{f_q-1} u^{\beta(\delta)(f_q +1 -f)} e_{a(q,f)}.\label{eqn:vu-def2}
\end{equation}
Next, assuming $f_q > 1$, let
\begin{equation}\label{eqn:vpu-defn}
v'_q(u) = \sum_{h=0}^{\delta} (-1)^{h} u^{B(\delta -h)} v_{\sigma^{h}(q)}(u).
\end{equation}
Note that $\delta(\sigma^{h}(q)) = \delta - h$, so~\eqref{eqn:vu-def2} yields
\begin{multline}\label{eqn:vpu-expand}
v_{\sigma^{h}(q)}(u) = 
e_{a(\sigma^{h}(q),1)} + u^{\beta(\delta - h)}e_{a(\sigma^{h - 1}(q), 1)} \\ {}+ \sum_{f=2}^{f_{\sigma^{h}(q)} - 1} u^{\beta(\delta - h)(f_{\sigma^{h}(q)} + 1 - f)} e_{a(\sigma^{h}(q), f)}.
\end{multline}
Expand~\eqref{eqn:vpu-defn} using~\eqref{eqn:vpu-expand} to get
\begin{multline}\label{eqn:vpu-def2}
v'_q(u) = \sum_{h=0}^{\delta} (-1)^{h} u^{B(\delta - h)} \Big(e_{a(\sigma^{h}(q), 1)} + 
u^{\beta(\delta - h)}e_{a(\sigma^{h - 1}(q), 1)}\Big) \\
+ \sum_{h=0}^{\delta} \sum_{f= 2}^{f_{\sigma^{h}(q)} - 1} (-1)^{h}
u^{B(\delta - h)}u^{\beta(\delta - h)(f_{\sigma^{h}(q)} + 1 - f)} e_{a(\sigma^{h}(q), f)}.
\end{multline}
Since $B(\delta - h) + \beta(\delta - h) = B(\delta - h +1)$, the first summation above is equal to
\[
\sum_{h=0}^{\delta} (-1)^{h} \Big( u^{B(\delta - h)} e_{a(\sigma^{h}(q), 1)} + 
u^{B(\delta - (h-1))} e_{a(\sigma^{h - 1}(q), 1)}\Big).
\]
Most terms in this summation cancel, and we can rewrite~\eqref{eqn:vpu-def2} as
\begin{multline}\label{eqn:vpu-cancel}
v'_q(u) = (-1)^{\delta} e_{a(M, 1)} + u^{B(\delta+1)} e_{a(\sigma^{-1}(q),1)} \\
+ \sum_{h=0}^{\delta} \sum_{f= 2}^{f_{\sigma^{h}(q)} - 1} (-1)^{h}
u^{B(\delta - h)}u^{\beta(\delta - h)(f_{\sigma^{h}(q)} + 1 - f)} e_{a(\sigma^{h}(q), f)}.
\end{multline}
In the special case where $q = m$, so that $\sigma^{-1}(q) = M$,~\eqref{eqn:vpu-cancel} simplifies to
\begin{multline}\label{eqn:vpm-cancel}
v'_m(u) = \Big((-1)^{C-1} + u^{B(C)}\Big) e_{a(M,1)} \\
+ \sum_{h=0}^{C-1} \sum_{f= 2}^{f_{\sigma^{h}(m)} - 1}
u^{B(C - 1 - h)}u^{\beta(C - 1 - h)(f_{\sigma^{h}(m)} + 1 - f)} e_{a(\sigma^{h}(m), f)}.
\end{multline}

For another description of $v'_q(u)$, separate the $h=0$ terms in~\eqref{eqn:vpu-cancel} and combine with $u^{B(\delta+1)}e_{a(\sigma^{-1}(q),1)} = u^{B(\delta)} u^{\beta(\delta)}e_{a(q,f_q)}$ to obtain
\begin{multline}\label{eqn:vpu-can2}
v'_q(u) = (-1)^{\delta} e_{a(M, 1)} + u^{B(\delta)} \Big(v_q(u) - e_{a(q,1)}\Big) \\
+ \sum_{h=1}^{\delta} \sum_{f= 2}^{f_{\sigma^{h}(q)} - 1} (-1)^{h}
u^{B(\delta - h)}u^{\beta(\delta - h)(f_{\sigma^{h}(q)} + 1 - f)} e_{a(\sigma^{h}(q), f)}.
\end{multline}

Finally, if $f_q > 1$ and $q \ne m$, let
\[
v''_q(u) = (1 + (-1)^{C-1}u^{B(C)})v'_q(u) + (-1)^{C+\delta}v'_m(u). 
\]
It will be important to rewrite this as
\begin{multline}\label{eqn:vppu-cancel}
v''_q(u) =  (1 + (-1)^{C-1}u^{B(C)})\big(v'_q(u) - (-1)^\delta e_{a(M,1)}\big) \\
+ 
(-1)^{C+\delta}\Big( v'_m(u) - \big((-1)^{C-1} + u^{B(C)}\big) e_{a(M,1)}\Big).
\end{multline}

\begin{lem}\label{lem:vu-terms}
Let $1 \le j \le n$. Suppose $f_q>1$.
\begin{enumerate}
\item If $e_j$ occurs with nonzero coefficient in $v_q(u)$, then $a(q,2) \le j \le a(M,1)$.
\item If $e_j$ occurs with nonzero coefficient in $v'_q(u)$ or $v''_q(u)$, then $j \le a(M,1)$.
\end{enumerate}
\end{lem}
\begin{proof}
Immediate from the definitions and Lemma~\ref{lem:aq-basic}.
\end{proof}

\subsection{Evaluation at \texorpdfstring{$u = 0$}{u=0}}
\label{ss:eval}

In this subsection, we retain the convention from~\eqref{eqn:q-fix}.  For a positive integer $N$, form the $n \times n$ matrix of Laurent polynomials
\[
\bar\theta_{i-1}(-u^{-N}) = 
\left[\begin{smallmatrix}
u^{-N} \\ & \ddots \\ && u^{-N} \\ &&& 1 \\ &&&& \ddots \\ &&&&& 1
\end{smallmatrix}\right]
\]
defined as in~\eqref{eqn:bartheta-defn}.  We will consider the vector $\bar\theta_{i-1}(-u^{-N})v_q(u)$ and related expressions.  Because these vectors involve Laurent polynomials, it may not be possible to evaluate them at $u = 0$. We will prove a large number of lemmas describing cases in which evaluation at $u = 0$ is possible.

\begin{lem}\label{lem:vu-final}
Suppose $\phi(q,i) = 1$. If $f_q=1$, suppose $i \le a(q,1)$. For $N > 0$, we have
\[
\big(\bar\theta_{i-1}(-u^{-N}) \cdot  v_q(u)\big)\big|_{u = 0} = e_{a(q,1)}.
\]
\end{lem}
\begin{proof}
If $f_q=1$, then $v_q(u)=e_{a(q,1)}$, so if $i\leq a(q,1)$, we have $\bar\theta_{i-1}(-u^{-N}) v_q(u) = e_{a(q,1)}$, whence the claim follows. Now suppose $f_q>1$. Since $2>\phi(q,i)$, we have $i \le a(q,2)$ by Lemma~\ref{lem:aq-basic}\eqref{it:aqf}, and then Lemma~\ref{lem:vu-terms} implies that $\bar\theta_{i-1}(-u^{-N}) v_q(u) = v_q(u)$.  The fact that $v_q(u)|_{u=0} = e_{a(q,1)}$ is clear from the definition.
\end{proof}

\begin{rmk}\label{rmk:vu-final-fqone}
If $f_q=1$ and $i>a(q,1)$, then $\bar\theta_{i-1}(-u^{-N}) v_q(u) = u^{-N}e_{a(q,1)}$ cannot be evaluated at $u=0$. In this case, the relevant statement we use later is 
\[
\big(\bar\theta_{i-1}(-u^{-N}) \cdot u^N v_q(u)\big)\big|_{u = 0} = e_{a(q,1)}.
\]
\end{rmk}

\begin{lem}\label{lem:vu-calc}
Suppose $2 \le \phi(q,i) \le f_q$.  For $N \gg 0$, we have
\[
\Big(\bar\theta_{i-1}(-u^{-N}) \cdot u^{N - \beta(\delta)(f_q+1-\phi(q,i))} (v_q(u) - e_{a(q,1)})\Big)\Big|_{u=0} 
= e_{a(q,\phi(q,i))}.
\]
\end{lem}
\begin{proof}
On the left-hand side, $e_{a(q,1)}$ does not appear.  For $f \ge 2$, the coefficient of $e_{a(q,f)}$ is as follows (using Lemma~\ref{lem:aq-basic}\eqref{it:aqf}):
\[
\begin{cases}
u^{\beta(\delta)(\phi(q,i) - f)} & \text{if $a(q,f) < i$, or equivalently, if $f \le \phi(q,i)$,} 
 \\
u^{N+ \beta(\delta)(\phi(q,i) - f)} & \text{if $a(q,f) \ge i$, or equivalently, if $f > \phi(q,i)$.}
\end{cases}
\]
If $N > - \beta(\delta)(\phi(q,i) - f)$ for all $f$, then these coefficients are always of the form $u^{\text{nonnegative integer}}$, with $u^0$ occurring only when $f = \phi(q,i)$.  The lemma follows.
\end{proof}

\begin{lem}\label{lem:vpu-calc}
Suppose $2 \le \phi(q,i) \le f_q$.  For $N \gg 0$, we have
\[
\Big(\bar\theta_{i-1}(-u^{-N}) \cdot u^{N - B(\delta) - \beta(\delta)(f_q + 1 - \phi(q,i))} (v'_q(u) - (-1)^{\delta} e_{a(M,1)})\Big)\big|_{u = 0}
= e_{a(q,\phi(q,i))}.
\]
\end{lem}
\begin{proof}
First, if $\delta=0$, then $q=M$ and $v'_q(u)=v_q(u)$, so the claim reduces to Lemma~\ref{lem:vu-calc}. Now suppose $1 \le h \le \delta$ and $2 \le f \le f_{\sigma^{h}(q)} - 1$. We will prove below that
\begin{equation}\label{eqn:vpu-calc-claim}
B(\delta - h) + \beta(\delta - h)(f_{\sigma^{h}(q)} + 1 - f) > B(\delta) + \beta(\delta)(f_q + 1 - \phi(q,i)).
\end{equation} This inequality implies
\begin{multline*}
\Big(\bar\theta_{i-1}(-u^{-N}) \cdot u^{N - B(\delta) - \beta(\delta)(f_q + 1 - \phi(q,i))}\\
{} \cdot u^{B(\delta - h)}u^{\beta(\delta - h)(f_{\sigma^{h}(q)} + 1 - f)} e_{a(\sigma^{h}(q), f)} \Big)\Big|_{u=0} = 0.
\end{multline*}
The lemma then follows by combining this observation with~\eqref{eqn:vpu-can2} and Lemma~\ref{lem:vu-calc}. 

Let us now prove~\eqref{eqn:vpu-calc-claim}. First, consider the case $h=1$. Using $B(\delta)=\beta(\delta-1)+B(\delta-1)$, the inequality becomes
\[
n^{r-\delta+1}(f_{\sigma(q)}+1-f)> n^{r-\delta+1} + n^{r-\delta}(f_q+1-\phi(q,i)).
\]
This can be rewritten as 
\[
n(f_M-1-f)+n> f_q+1-\phi(q,i).
\]
Under our assumptions, $f_M-1-f\ge 0$ and $f_q+1-\phi(q,i)\le f_q-1\le n$, so we at least have 
\[
n(f_M-1-f)+n \ge f_q+1-\phi(q,i).
\]
It remains to show that equality is not possible. For equality to hold, it is necessary that $f_q=n+1$. By Lemma~\ref{lem:aq-basic}\eqref{it:fqnr}, we then have $\sigma(q)=q$. This contradicts $\delta>0$, so~\eqref{eqn:vpu-calc-claim} holds in this case.

Now assume $h>1$. Under our assumptions, $f_{\sigma^{h}(q)} + 1 - f \ge 2$ and $f_q + 1 - \phi(q,i) \le f_q -1 \le n$.  It is therefore enough to prove that
\[
B(\delta - h) + 2\beta(\delta - h) > B(\delta) + n\beta(\delta).
\]
If $h=\delta=2$, this becomes 
\[
2n^r > n^r + 2n^{r-1},
\]
which follows from $2=\delta<r<n$. If $h=\delta>2$, the inequality becomes 
\[
2n^r > n^r + n^{r-1} + \cdots + 2n^{r-\delta+1},
\]
which is clear by interpreting both sides in base $n$ and using $n>r>\delta>2$. If $h=\delta-1$, then the inequality becomes 
\[
n^r + 2n^{r-1} > n^r + n^{r-1} + \cdots + 2n^{r-\delta+1},
\]
which is again clear in base $n$.

Now assume $1<h<\delta-1$. A straightforward calculation shows that the inequality is equivalent to 
\begin{equation}\label{eqn:base3-calc}
n^{r-\delta + h} > n^{r-\delta + h-1} + n^{r-\delta + h-2} + \cdots + n^{r-\delta+2} + 2n^{r-\delta+1}.
\end{equation}
As $n>r>\delta>1+h>2$, the above inequality is an obvious statement in base $n$. Therefore,~\eqref{eqn:vpu-calc-claim} holds.
\end{proof}

\begin{lem}\label{lem:vp-final}
Suppose $2 \le \phi(q,i) \le f_q$ and $i \le a(M,1)$. For $N \gg 0$, we have
\[
\Big(\bar\theta_{i-1}(-u^{-N}) \cdot u^{N - B(\delta) - \beta(\delta)(f_q + 1 - \phi(q,i))}v'_q(u)\Big)\Big|_{u=0} = e_{a(q,\phi(q,i))}.
\]
\end{lem}
\begin{proof}
Since $i \le a(M,1)$, we have
\begin{multline*}
\big(\bar\theta_{i-1}(-u^{-N}) \cdot u^{N - B(\delta) - \beta(\delta)(f_q + 1 - \phi(q,i))} e_{a(M,1)}\big)\big|_{u=0} \\
=\big(u^{N - B(\delta) - \beta(\delta)(f_q + 1 - \phi(q,i))} e_{a(M,1)}\big)\big|_{u=0} = 0.
\end{multline*}
The result then follows from Lemma~\ref{lem:vpu-calc}.
\end{proof}

\begin{lem}\label{lem:vpm-final}
Suppose $f_q>1$ and $i > a(M,1)$.  For $N > 0$, we have
\[
\big(\bar\theta_{i-1}(-u^{-N}) \cdot (-1)^{C-1}u^N v'_m(u)\big)\big|_{u=0} = e_{a(m,\phi(m,i))}.
\]
\end{lem}
\begin{proof}
By Lemma~\ref{lem:vu-terms}, $\bar\theta_{i-1}(-u^{-N})\cdot (-1)^{C-1} u^N v'_m(u) = (-1)^{C-1}v'_m(u)$.  It is clear from~\eqref{eqn:vpm-cancel} that $(-1)^{C-1}v'_m(u)|_{u=0} = e_{a(M,1)}$.  Finally, $a(M,1) = a(m,f_m) = a(m,\phi(m,i))$ by~\eqref{eqn:aq-sigma} and Lemma~\ref{lem:aq-basic}\eqref{it:Mphi}.
\end{proof}

\begin{lem}\label{lem:vmd-calc}
Suppose $f_q > 1$. For $N > 0$, we have
\[
\Big(\bar\theta_{i-1}(-u^{-N}) \cdot u^{N - B(\delta+1)} \big( v'_m(u) - \big((-1)^{C-1} + u^{B(C)}\big) e_{a(M,1)}\big) \Big)\Big|_{u=0} = 0.
\]
\end{lem}
\begin{proof}
We will show below that if $0 \le h \le C-1$ and $2 \le f \le f_{\sigma^{h}(m)} - 1$, then
\begin{equation}\label{eqn:vmd-calc-claim}
B(C - 1 - h) + \beta(C - 1 - h)(f_{\sigma^{h}(m)} + 1 - f) > B(\delta+1).
\end{equation}
This claim implies that
\[
\Big(\bar\theta_{i-1}(-u^{-N}) \cdot u^{N - B(\delta+1)} \cdot
u^{B(C-1 - h)}u^{\beta(C-1 - h)(f_{\sigma^{h}(m)} + 1 - f)} e_{a(\sigma^{h}(m), f)} \Big)\Big|_{u=0} = 0.
\]
The lemma follows by combining this with~\eqref{eqn:vpm-cancel}.

Let us prove~\eqref{eqn:vmd-calc-claim}.  Under our assumptions, $f_{\sigma^{h}(m)} + 1 - f \ge 2$, so it is enough to prove that $B(C - 1 - h) + 2\beta(C - 1 - h) > B(\delta+1)$, or
\[
\beta(0) + \cdots + \beta(C-h-2) + 2\beta(C-h-1) > \beta(0) + \cdots + \beta(\delta).
\]
If $\delta \le C - h - 1$, this is obvious.  If $\delta > C-h-1$, our claim reduces to
\[
n^{r - C + h + 1} > n^{r - C + h} + n^{r - C +h -1} + \cdots + n^{r-\delta},
\]
and this holds for any $n \ge 2$.  Thus,~\eqref{eqn:vmd-calc-claim} is proved.
\end{proof}

\begin{lem}\label{lem:vpp-final}
Suppose $f_q>1$ and $q \ne m$. For $N \gg 0$, we have
\[
\Big(\bar\theta_{i-1}(-u^{-N}) \cdot u^{N - B(\delta+1)}v''_q(u)\Big)\Big|_{u=0} = e_{a(q,\phi(q,i))}.
\]
\end{lem}
\begin{proof}
Since $q \ne m$, we have $\delta < C-1$, so $B(\delta+1) < B(C)$. Thus
\begin{equation}\label{eqn:vpp-final-pre}
\big(\bar\theta_{i-1}(-u^{-N}) \cdot u^{N-B(\delta + 1)} u^{B(C)} (v'_q(u) - (-1)^\delta e_{a(M,1)})\big)\big|_{u = 0} = 0.
\end{equation}

Now expand $v''_q(u)$ using~\eqref{eqn:vppu-cancel}.  The lemma follows from (the $\phi(q,i) = f_q$ case of) Lemma~\ref{lem:vpu-calc}, Lemma~\ref{lem:vmd-calc}, and~\eqref{eqn:vpp-final-pre}.
\end{proof}

\begin{thm}\label{thm:main-fund}
If $1 \le t \le n-1$, then $X_{\varpi_t} = \bGr_{\varpi_t}$.
\end{thm}
\begin{proof}
Let $r = n -t$, and let $x$ be a $\varpi_{n-r}$-permissible alcove.  Carry out the constructions of~\S\S\ref{ss:fundalcove}--\ref{ss:families} with respect to this alcove to obtain a sequence of vectors $v_1(u), \ldots, v_r(u)$.  

For each pair $(q,i)$ with $1 \le q \le r$ and $1 \le i \le n+1$, one may invoke the lemmas in~\S\ref{ss:eval}.  Some of these involve a condition on a positive integer $N$.  Choose $N$ to sufficiently large for all of these lemmas for all pairs $(q,i)$.  Then, define a vector $\hat v^{(i)}_q(u)$ as follows:
\[
\hat v_q^{(i)}(u) = \hbox{\small$
\begin{cases}
u^N v_q(u) & \text{if $f_q=1$ and $i>a(q,1)$,} \\
v_q(u) & \text{if $\phi(q,i) = 1$, and $f_q>1$ or $i\le a(q,1)$,} \\
u^{N - B(\delta) - \beta(\delta)(f_q + 1 - \phi(q,i))} v'_q(u) & \text{if $2 \le \phi(q,i) \le f_q$ and $i \le a(M,1)$,} \\
(-1)^{C-1}u^N v'_q(u) & \text{if $f_q>1$ and $i > a(M,1)$ and $q = m$,} \\
u^{N - B(\delta+1)} v''_q(u) & \text{if $f_q>1$ and $i > a(M,1)$ and $q \ne m$.}
\end{cases}$}
\]
(By Lemma~\ref{lem:aq-basic}\eqref{it:Mphi}, the five cases above are mutually exclusive.)  Let $A^{(i)}(u)$ be the $n \times r$ matrix whose columns are $\hat v^{(i)}_1(u), \ldots, \hat v^{(i)}_r(u)$.  In particular, the columns of $A^{(1)}(u)$ are $v_1(u), \ldots, v_r(u)$.

From the definitions of $v'_q(u)$ and $v''_q(u)$, we see that in all cases, $\hat v^{(i)}_q$ is a linear combination (with coefficients in $\Z[u]$) of $v_1(u), \ldots, v_r(u)$.  In other words, there exists an $r \times r$ matrix $B^{(i)}(u)$ such that 
\begin{equation}\label{eqn:a-chgbasis}
A^{(i)}(u) = A^{(1)}(u) B^{(i)}(u) \qquad \text{for $2 \le i \le n+1$.}
\end{equation}
By Remark~\ref{rmk:vu-final-fqone} and Lemmas~\ref{lem:vu-final}, \ref{lem:vp-final}, \ref{lem:vpm-final}, and~\ref{lem:vpp-final}, we have
\[
\big( \bar\theta_{i-1}(-u^{-N}) A^{(i)}(u) \big) \big|_{u=0} = \begin{array}{c}
\text{the matrix with columns}\\
e_{a(1,\phi(1,i))}, e_{a(2,\phi(2,i))}, \ldots, e_{a(r,\phi(r,i))}
\end{array}.
\]
From the definitions, we have $a(q,\phi(q,i)) = j^{(i)}_q$, so
\[
\big( \bar\theta_{i-1}(-u^{-N}) A^{(i)}(u) \big) \big|_{u=0} = \text{the matrix with columns $\{ e_j \mid j \in J^{(i)}\}$.}
\]
In particular, the columns of $( \bar\theta_{i-1}(-u^{-N}) A^{(i)}(u) ) |_{u=0}$ are linearly independent.  Their span determines a point of $G(n,r)$, and then a point of $\Gr_{\varpi_{n-r}}$ via the isomorphism of Example~\ref{ex:locfund}.  Indeed, from~\eqref{eqn:alcove-J}, we see that
\begin{equation}\label{eqn:spn-calc}
\spn ( \bar\theta_{i-1}(-u^{-N}) A^{(i)}(u) ) |_{u=0} \in G(n,r)
\quad\text{corresponds to}\quad
\bL^{x^{(i)} - \varpi_{i-1}} \in \Gr_{\varpi_{n-r}}.
\end{equation}
(Here, and below, the ``span'' of a matrix means its column span.)

Let us now specialize $u$ to other elements in $\C$.  There is a Zariski open subset $U \subset \C$ containing $0$ such that for $z \in U$, the matrices
\[
(  A^{(1)}(u) \big) \big|_{u=z}, ( \bar\theta_{1}(-u^{-N}) A^{(1)}(u) \big) \big|_{u=z}, \ldots, ( \bar\theta_{n-1}(-u^{-N}) A^{(n)}(u) \big) \big|_{u=z}
\]
all have rank $r$.  Define a map $\mathbf{A}: U \to \C \times G(n,r) \times \cdots \times G(n,r)$ by
\begin{multline*}
\mathbf{A}(z) = \Big(-z^{N}, \spn(  A^{(1)}(u) \big) \big|_{u=z}, \spn( \bar\theta_{1}(-u^{-N}) A^{(1)}(u) \big) \big|_{u=z}, \ldots, \\\spn( \bar\theta_{n-1}(-u^{-N}) A^{(n)}(u) \big) \big|_{u=z}\Big)
.
\end{multline*}

Identifying $G(n,r)$ with $\Gr_{\varpi_{n-r}} \subset \Gr$, we may regard $\mathbf{A}$ as a map $U \to \C \times \Gr \times \cdots \times \Gr$.  
When $z \in U \cap \C^\times$, it follows from~\eqref{eqn:a-chgbasis} that
\[
\spn A^{(1)}(z) = \bar\theta_{i-1}(-z^N) \cdot \spn \bar\theta_{i-1}(-z^{-N})A^{(i)}(z) \qquad\text{for $1 \le i \le n+1$.}
\]
In other words, for $z \in U \cap \C^\times$, $\mathbf{A}(z)$ is lattice family in $\overline{\bGr_{\varpi_{n-r}}} \cap \bGr_\gen$.  Passing to closures, we deduce that $\mathbf{A}(0)$ also lies in $\overline{\bGr_{\varpi_{n-r}}}$.  By~\eqref{eqn:spn-calc},
\[
\mathbf{A}(0) = (0,  \bL^{x^{(1)}}, \bL^{x^{(2)} - \varpi_1}, \ldots, \bL^{x^{(n)} - \varpi_{n-1}}),
\]
and this corresponds under the isomorphism of Lemma~\ref{lem:bgr-gen-spc} to the lattice chain $\ubL^x \in \Fl$.  We have established~\eqref{eqn:main-claim}, and the theorem follows.
\end{proof}

\section{Alcoves and dominant coweights}
\label{sec:alcoves}

Let $\lambda \in \bX^+$, and let $x = (x^{(1)}, \ldots, x^{(n)})$ be a $\lambda$-permissible alcove.  By definition, this means that $\dom(x^{(i)} - \varpi_{i-1}) \preceq \lambda$ for each $i$.  That is, there are permutations $\delta_1, \ldots, \delta_n \in \fS_n$ such that
\[
\delta_i(x^{(i)} - \varpi_{i-1}) \preceq \lambda.
\]
Most of the work in this section is devoted to showing that each $\delta_i$ can be chosen in such a way that it differs ``minimally'' from its neighbors $\delta_{i-1}$ and $\delta_{i+1}$.  For a precise statement, see Lemma~\ref{lem:to-perm}.  As an application, we will prove a result about alcoves (Theorem~\ref{thm:convolve-comb}) that is a combinatorial precursor to the convolution geometry that will be studied in Section~\ref{sec:conv}.

\subsection{A lemma on dominant coweights}

The following fact is undoubtedly well known, but we include a proof as we were unable to find a reference for it.

\begin{lem}\label{lem:dom-subtract}
Let $\mu, \lambda \in \bX^+$ be dominant coweights such that $\lambda = \mu + \varpi_t$, where $1 \le t \le n-1$.  If $\lambda' \in \bX^+$ is another dominant coweight satisfying $\lambda' \preceq \lambda$, then $\dom(\lambda' - \varpi_t) \preceq \mu$.
\end{lem}
\begin{proof}
It is immediate from the assumptions that
\begin{equation}\label{eqn:ds-easy}
\lambda' - \varpi_t \preceq \lambda - \varpi_t = \mu.
\end{equation}
If $\lambda'_t > \lambda'_{t+1}$, then $\lambda' - \varpi_t$ is already dominant, so there is nothing to prove.

For the remainder of the proof, we assume that $\lambda'_t = \lambda'_{t+1}$.  Let $m = \lambda'_t$.  Let $a$ be the smallest integer such that $\lambda'_a = m$, and let $b$ be the largest integer such that $\lambda'_b = m$.  Thus, $a \le t < t+1 \le b$, and
\[
\lambda' = (\lambda'_1, \lambda'_2, \ldots , \lambda'_{a-1}, m, m , \ldots, m , \lambda'_{b+1} , \lambda'_{b+2} , \ldots , \lambda'_n)
\]
with
\[
\lambda'_1 \ge \lambda'_2 \ge \cdots \ge \lambda'_{a-1} >  m > \lambda'_{b+1} \ge \lambda'_{b+2} \ge \cdots \ge \lambda'_n.
\]
We have
\[
\lambda' - \varpi_t =
(\lambda'_1-1,  \ldots , \lambda'_{a-1}-1, \underbrace{m-1, \ldots, m-1}_{\text{$t-a+1$ terms}}, \underbrace{m,\ldots,m}_{\text{$b-t$ terms}}, \lambda'_{b+1},\ldots, \lambda'_n),
\]
and hence
\[
\dom(\lambda' - \varpi_t) = 
(\lambda'_1-1,  \ldots , \lambda'_{a-1}-1, \underbrace{m,\ldots,m}_{\text{$b-t$ terms}}, \underbrace{m-1, \ldots, m-1}_{\text{$t-a+1$ terms}},  \lambda'_{b+1},\ldots, \lambda'_n).
\]
Our goal is to prove that $\dom(\lambda' - \varpi_t) \preceq \mu$, or
\begin{equation}\label{eqn:ds-goal}
\sum_{i=1}^k \dom(\lambda' - \varpi_t)_i \le \sum_{i=1}^k \mu_i \qquad\text{for $1 \le k \le n$,}
\end{equation} with equality when $k=n$. The equality case follows from $|\dom(\lambda'-\varpi_t)|=|\lambda'-\varpi_t|= |\lambda-\varpi_t|=|\mu|$.
Before proving the inequalities, we set up some additional notation.  Since $\lambda = \mu + \varpi_t$, we have
\[
\lambda_i = 
\begin{cases}
\mu_i+1 & \text{if $1 \le i \le t$,} \\
\mu_i & \text{if $t+1 \le i \le n$.}
\end{cases}
\]
Since $\mu$ is dominant, this implies that
\begin{equation}\label{eqn:ds-lambda}
\lambda_1 \ge \cdots \ge \lambda_t > \lambda_{t+1} \ge \cdots \ge \lambda_n.
\end{equation}

For $1 \le i \le n$, let $f_i = \lambda_i - \lambda'_i$.  Since $\lambda' \preceq \lambda$, we have
\begin{equation}\label{eqn:ds-fdom}
\sum_{i=1}^k f_i \ge 0
\qquad\text{for $1 \le k \le n$.}
\end{equation}
Since $\lambda'_a = \lambda'_{a+1} = \cdots = \lambda'_b$, it follows from~\eqref{eqn:ds-lambda} that
\begin{equation}\label{eqn:ds-fin}
f_a \ge f_{a+1} \ge \cdots \ge f_t > f_{t+1} \ge f_{t+2} \ge \cdots \ge f_b.
\end{equation}

We are now ready to prove~\eqref{eqn:ds-goal}.  We break up the problem into cases as follows.

\textit{Case 1. $1 \le k \le a-1$.} In this case,~\eqref{eqn:ds-goal} simplifies to
\[
\sum_{i=1}^k (\lambda'_i - 1) \le \sum_{i=1}^k (\lambda_i - 1)
\qquad\text{or, equivalently,}\qquad
\sum_{i=1}^k \lambda'_i - k \le \sum_{i=1}^k \lambda_i - k,
\]
and this holds because $\lambda' \preceq \lambda$.

\textit{Case 2. $a \le k \le b$.}
Assuming $1 \le i \le b$, we have
\[
\dom(\lambda' - \varpi_t)_i =
\begin{cases}
\lambda'_i -1 & \text{if $1 \le i \le a-1$ or $b-t+a \le i \le b$,} \\
\lambda'_i & \text{if $a \le i \le b-t+a-1$,}
\end{cases}
\]
so the left-hand side of~\eqref{eqn:ds-goal} simplifies to
\begin{multline*}
\sum_{i=1}^k \dom(\lambda' - \varpi_t)_i = \sum_{i=1}^k (\lambda'_i - 1) + |\{ i \mid \text{$1 \le i \le k$ and $a \le i \le b-t+a-1$} \}| \\
= \sum_{i=1}^k \lambda'_i - k + \min \{k, b-t+a-1\} - a + 1 
= \sum_{i=1}^k \lambda'_i + \min \{-a+1, b-t-k\}.
\end{multline*}
Therefore,~\eqref{eqn:ds-goal} simplifies to
\[
\sum_{i=1}^k \lambda'_i + \min \{-a+1, b-t-k\} \le \sum_{i=1}^k \mu_i = \sum_{i=1}^k \lambda_i - \min\{ t, k\}.
\]
Since
\begin{multline*}
\min\{t,k\} + \min \{-a+1, b-t-k\} = \min \{ t-a+1, k-a+1, b-t, b-k\},
\end{multline*}
we see that~\eqref{eqn:ds-goal} is equivalent to
\begin{equation}\label{eqn:ds-goal2}
\sum_{i=1}^k f_i \ge \min \{ t-a+1, k-a+1, b-t, b-k\}.
\end{equation}
To prove this, we consider various subcases.

\textit{Case 2a. $k \le t$ and $f_k \ge 1$.}
Using~\eqref{eqn:ds-fdom} and~\eqref{eqn:ds-fin}, we have
\[
\sum_{i=1}^k f_i = \sum_{i=1}^{a-1} f_i + f_a + \cdots + f_k 
\ge \underbrace{\sum_{i=1}^{a-1} f_i}_{{}\ge 0}{} + (k-a+1)\underbrace{f_k}_{{}\ge 1}{} \ge k - a+1.
\]

\textit{Case 2b. $k \le t$ and $f_k \le 0$.}
By~\eqref{eqn:ds-fin}, we have $0 \ge f_{k+1} \ge \cdots \ge f_t$, and $0 > f_{t+1}$.  Since $\sum_{i=1}^b f_i \ge 0$ by~\eqref{eqn:ds-fdom}, we have
\begin{multline*}
\sum_{i=1}^k f_i \ge - \sum_{i=k+1}^b f_i
= - \sum_{i=k+1}^t f_i - f_{t+1} - f_{t+2} - \cdots - f_b \\
\ge \underbrace{- \sum_{i=k+1}^t f_i}_{{}\ge 0}{} + (b-t)\underbrace{(-f_{t+1})}_{{}\ge 1} \ge b-t.
\end{multline*}

\textit{Case 2c. $k > t$ and $f_k \ge 0$.}
By~\eqref{eqn:ds-fin}, we have $f_a \ge \cdots \ge f_t > 0$ and $f_{t+1} \ge \cdots \ge f_k \ge 0$, so
\[
\sum_{i=1}^k f_i = \underbrace{\sum_{i=1}^{a-1} f_i}_{{}\ge 0}{} + \underbrace{\sum_{i=a}^t \underbrace{f_i}_{{}\ge 1}}_{{}\ge t-a+1}{} + \underbrace{\sum_{i=t+1}^k f_i}_{{}\ge 0}{} \ge t - a+1.
\]

\textit{Case 2d. $k > t$ and $f_k \le -1$.}
We have $\sum_{i=1}^b f_i \ge 0$, and hence
\[
\sum_{i=1}^k f_i \ge \sum_{i=k+1}^b \underbrace{(-f_i)}_{{}\ge 1} \ge b-k.
\]

\textit{Case 3. $b+1 \le k \le n$.}
In this case, we have
\[
\sum_{i=1}^k \dom(\lambda' - \varpi_t)_i = \sum_{i=1}^k (\lambda' - \varpi_t)_i,
\]
so~\eqref{eqn:ds-goal} holds as a consequence of~\eqref{eqn:ds-easy}.
\end{proof}

\subsection{Rotation of alcoves}

Let $\cox \in \fS_n$ be the $n$-cycle $(1\,2\, \cdots\, n)$, i.e.,
\[
\cox(i) =
\begin{cases}
i+1 & \text{if $1 \le i \le n-1$,} \\
1 & \text{if $i = n$.}
\end{cases}
\]
Next, let $x = (x^{(1)}, \ldots, x^{(n)})$ be an alcove.  We define a new sequence of coweights $\rot(x) = (\rot(x)^{(1)}, \rot(x)^{(2)}, \ldots, \rot(x)^{(n)})$ by the formula
\[
\rot(x)^{(i)} = 
\begin{cases}
\cox^{-1}(x^{(i+1)} - \varpi_1) & \text{if $1 \le i \le n-1$,} \\
\cox^{-1}(x^{(1)}) + \varpi_{n-1} & \text{if $i = n$.}
\end{cases}
\]
The next two lemmas follow from the definitions; we omit their proofs.

\begin{lem}\label{lem:rot-basic}
Let $x$ be an alcove.
\begin{enumerate}
\item The sequence $\rot(x)$ is an alcove.
\item If $x$ is $\lambda$-permissible for some $\lambda \in \bX^+$, then $\rot(x)$ is also $\lambda$-permissible.
\item For $1 \le j \le n$, we have
\[
\rot^j(x)^{(i)} = 
\begin{cases}
\cox^{-j}(x^{(i+j)} - \varpi_j) & \text{if $1 \le i \le n-j$,} \\
\cox^{-j}(x^{(i-n+j)} - \varpi_{i-n+j-1}) + \varpi_{i-1} & \text{if $n-j < i \le n$.}
\end{cases}
\]
\item We have $\rot^n(x) = x$.
\end{enumerate}
\end{lem}

\begin{lem}
For any alcove $x$, we have $\sP(\rot(x)) = \cox^{-1} \sP(x) \cox$.
\end{lem}

\subsection{Partial orders on indices}

Let $x$ be an alcove.  We define a partial order $\trianglelefteq_x$ on the set $\{1, \ldots, n\}$ by the rule shown in Table~\ref{tab:po-defn}.

\begin{table}
For $i, j \in \{1, \ldots, n\}$, we declare that $i \trianglelefteq_x j$ if one of the following holds:
\[
\begin{array}{lll@{\text{and }}ll@{\text{and }}ll@{\text{and }}l}
\hline
(1) & x^{(1)}_i > x^{(1)}_j \\
(2) & x^{(1)\strut}_i = x^{(1)}_j && \sP(x)^{-1}(i) < i && \sP(x)^{-1}(j) < j && i \ge j \\
(3) & x^{(1)\strut}_i = x^{(1)}_j && \sP(x)^{-1}(i) < i && \sP(x)^{-1}(j) \ge j \\
(4) & x^{(1)\strut}_i = x^{(1)}_j && \sP(x)^{-1}(i) = i && \sP(x)^{-1}(j) = j && i = j \\
(5) & x^{(1)\strut}_i = x^{(1)}_j && \sP(x)^{-1}(i) = i && \sP(x)^{-1}(j) > j \\
(6) & x^{(1)\strut}_i = x^{(1)}_j && \sP(x)^{-1}(i) > i && \sP(x)^{-1}(j) > j && i \ge j \\
\hline
\end{array}
\]
\caption{Definition of $\trianglelefteq_x$}\label{tab:po-defn}
\end{table}

\begin{prop}\label{prop:po-compare}
Let $x$ be an alcove, and let $i, j \in \{1, \ldots,n\}$.
\begin{enumerate}
\item If $\sP(x)(1) = 1$, then $i \trianglelefteq_x j$ if and only if $\cox^{-1}(i) \trianglelefteq_{\rot(x)} \cox^{-1}(j)$.\label{it:po-triv}
\item If $i, j \ne \sP(x)(1)$, then $i \trianglelefteq_x j$ if and only if $\cox^{-1}(i) \trianglelefteq_{\rot(x)} \cox^{-1}(j)$.\label{it:po-iff}
\item If $\sP(x)(1) \trianglelefteq_x j$, then $\cox^{-1}(\sP(x)(1)) \trianglelefteq_{\rot(x)} \cox^{-1}(j)$.\label{it:po-move}
\end{enumerate}
\end{prop}
\begin{proof}
\eqref{it:po-triv} Assume first that $\sP(x)(1) = 1$.   In this case, $x^{(2)} = x^{(1)} + \varpi_1$, so it follows from the definition that $\rot(x)^{(1)} = \cox^{-1}(x^{(1)})$, so for $i\in\{1,\dots,n\}$, we have
\begin{equation}\label{eqn:po1-rot}
\rot(x)^{(1)}_{\cox^{-1}(i)} = x^{(1)}_{i}.
\end{equation}
Next, we claim that for all $i$, we have
\begin{equation}\label{eqn:po1-sp}
\sP(\rot(x))^{-1}(\cox^{-1}(i)) < \cox^{-1}(i)
\qquad\text{if and only if}\qquad
\sP(x)^{-1}(i) < i,
\end{equation}
and likewise with ``$<$'' replaced by ``$=$'' or ``$>$.''  Since $\sP(\rot(x))^{-1}(\cox^{-1}(i)) = \cox^{-1}(\sP(x)^{-1}(i))$, our claim is equivalent to the assertion that
\[
\cox^{-1}(\sP(x)^{-1}(i)) < \cox^{-1}(i)
\qquad\text{if and only if}\qquad
\sP(x)^{-1}(i) < i.
\]
If $2 \le i \le n$, this assertion is obvious.  If $i = 1$, then the assumption that $\sP(x)(1) = 1$ means that both inequalities above are false, so the ``if and only if'' assertion is true.  Similar reasoning applies to the ``$=$'' and ``$>$'' cases.

Combining~\eqref{eqn:po1-rot},~\eqref{eqn:po1-sp}, and the definition of $\trianglelefteq_x$, we see that conditions (1), (3), (4), or (5) of Table~\ref{tab:po-defn} apply to $i \trianglelefteq_x j$ if and only if the same condition applies to $\cox^{-1}(i) \trianglelefteq_{\rot(x)} \cox^{-1}(j)$.  Conditions (2) and (6) are trickier because they involve the condition $i \ge j$.  However, these conditions also require $\sP(x)^{-1}(i) \ne i$ and $\sP(x)^{-1}(j) \ne j$, which means that $i, j \ne 1$.  For $i,j \in \{2, \ldots, n\}$, we have $i \ge j$ if and only if $\cox^{-1}(i) \ge \cox^{-1}(j)$.  Thus, in all cases, $i \trianglelefteq_x j$ if and only if $\cox^{-1}(i) \trianglelefteq_{\rot(x)} \cox^{-1}(j)$.

\eqref{it:po-iff}
Let $r = \sP(x)(1)$.  If $r = 1$, then this part of the proposition is subsumed by part~\eqref{it:po-triv}.  Assume for the remainder of this part of the proof that $r > 1$. Then
\begin{equation}\label{eqn:po-nontriv}
\rot(x)^{(1)}_{\cox^{-1}(i)} =
\begin{cases}
x^{(1)}_i & \text{if $i \ne 1,r$,} \\
x^{(1)}_1 -1 &\text{if $i = 1$,} \\
x^{(1)}_r + 1 & \text{if $i = r$.}
\end{cases}
\end{equation}
We consider various cases as follows.

\textit{Case 1. $i, j \ne 1, r$.}  Under this assumption, we have $\rot(x)^{(1)}_{\cox^{-1}(i)} = x^{(1)}_i$ and $\rot(x)^{(1)}_{\cox^{-1}(j)} = x^{(1)}_j$. The reasoning in the proof of part~\eqref{it:po-triv} applies verbatim to show that $i \trianglelefteq_x j$ if and only if $\cox^{-1}(i) \trianglelefteq_{\rot(x)} \cox^{-1}(j)$.

\textit{Case 2.  $i = 1$, $j \ne 1, r$, and $x^{(1)}_1 = x^{(1)}_j$.}  Our assumptions imply that $\sP(x)^{-1}(1) \ne 1$, and hence $\sP(x)^{-1}(1) > 1$, and also that $1 < j$.  So $i = 1$ and $j$ cannot satisfy any of the conditions in Table~\ref{tab:po-defn}; we necessarily have $1 \ntrianglelefteq_x j$.  On the other hand, we have
\[
\rot(x)^{(1)}_n = x^{(1)}_1 - 1 < x^{(1)}_j = \rot(x)^{(1)}_{\cox^{-1}(j)},
\]
so we also necessarily have $n \ntrianglelefteq_{\rot(x)} \cox^{-1}(j)$, as desired.

\textit{Case 3.   $i = 1$, $j \ne 1, r$, and $x^{(1)}_1 < x^{(1)}_j$.}
In this case, we again have $1 \ntrianglelefteq_x j$ and $n \ntrianglelefteq_{\rot(x)} \cox^{-1}(j)$.

\textit{Case 4.  $i = 1$, $j \ne 1, r$, and $x^{(1)}_1 > x^{(1)}_j$.}
In this case, we necessarily have $1 \trianglelefteq_x j$, so we must prove that $n \trianglelefteq_{\rot(x)} \cox^{-1}(j)$.  If $x^{(1)}_1 - x^{(1)}_j \ge 2$, then 
\[
\rot(x)^{(1)}_n = x^{(1)}_1 - 1 > x^{(1)}_j = \rot(x)^{(1)}_{\cox^{-1}(j)},
\]
and we are done.  On the other hand, suppose $x^{(1)}_1 - x^{(1)}_j = 1$, so that
\[
\rot(x)^{(1)}_n = \rot(x)^{(1)}_{\cox^{-1}(j)}.
\]
We have $\sP(\rot(x))^{-1}(n) \ne n$ and hence $\sP(\rot(x))^{-1}(n) < n$.  We obviously have $n > \cox^{-1}(j)$, so $n\trianglelefteq_{\rot(x)} \cox^{-1}(j)$ holds by either condition (2) or (3) in Table~\ref{tab:po-defn}.

\textit{Case 5. $i\ne 1,r$, $j = 1$, and $x^{(1)}_i = x^{(1)}_1$.}
Our assumptions imply that $\sP(x)^{-1}(1) > 1$ and $i > 1$.  Depending on how $\sP(x)^{-1}(i)$ compares with $i$, we see that one of conditions (3), (5), or (6) in Table~\ref{tab:po-defn} must hold, and we necessarily have $i \trianglelefteq_x j$.   On the other hand, we have
\[
\rot(x)^{(1)}_{\cox^{-1}(i)} = x^{(1)}_i > x^{(1)}_1 - 1 = \rot(x)^{(1)}_n,
\]
so by condition (1) of Table~\ref{tab:po-defn}, we necessarily have $\cox^{-1}(i) \trianglelefteq_{\rot(x)} n$.

\textit{Case 6. $i\ne 1,r$, $j = 1$, and $x^{(1)}_i > x^{(1)}_1$.}
In this case, we again have $i \trianglelefteq_x j$ and  $\cox^{-1}(i) \trianglelefteq_{\rot(x)} n$.

\textit{Case 7. $i\ne 1,r$, $j = 1$, and $x^{(1)}_i < x^{(1)}_1$.}
In this case, we have $i \ntrianglelefteq_x 1$, so we must prove that $\cox^{-1}(i) \ntrianglelefteq_{\rot(x)} n$.  If $x^{(1)}_1 - x^{(1)}_i \ge 2$, then we have
\[
\rot(x)^{(1)}_{\cox^{-1}(i)} = x^{(1)}_i < x^{(1)}_1 - 1 = \rot(x)^{(1)}_n,
\]
so $\cox^{-1}(i) \ntrianglelefteq_{\rot(x)} n$ as desired.  On the other hand, suppose $x^{(1)}_1 - x^{(1)}_i = 1$, so that $\rot(x)^{(1)}_{\cox^{-1}(i)} = \rot(x)^{(1)}_n$. Since $\sP(x)^{-1}(1) \ne 1$, we have $\sP(\rot(x))^{-1}(n) \ne n$, and hence $\sP(\rot(x))^{-1}(n) < n$.  We also have $\cox^{-1}(i) < n$, so none of conditions (2)--(6) in Table~\ref{tab:po-defn} apply, and we conclude that $\cox^{-1}(i) \ntrianglelefteq_{\rot(x)} n$.

\textit{Case 8. $i = j$.}
This case is trivial.

\textit{Conclusion.} Cases 1--8 cover all possible situations in part~\eqref{it:po-iff}, so this part of the proposition is now proved.

\eqref{it:po-move} We again let $r = \sP(x)(1)$, and we assume that $r > 1$.  The formula in~\eqref{eqn:po-nontriv} remains valid.  This time, we are only proving one implication, not an ``if and only if'' statement.  Assume $j$ is such that $r \trianglelefteq_x j$.  We may further assume without loss of generality that $j \ne r$.  From Table~\ref{tab:po-defn}, we must have $x^{(1)}_r \ge x^{(1)}_j$.  Since
\[
\rot(x)^{(1)}_{r-1} = x^{(1)}_r + 1
\qquad\text{and}\qquad
\rot(x)^{(1)}_{\cox^{-1}(j)} = 
\begin{cases}
x^{(1)}_j & \text{if $j \ne 1$,} \\
x^{(1)}_j -1 & \text{if $j = 1$,}
\end{cases}
\]
we have $\rot(x)^{(1)}_{r-1} > \rot(x)^{(1)}_{\cox^{-1}(j)}$ in all cases, and hence $r-1 \trianglelefteq_{\rot(x)} \cox^{-1}(j)$.
\end{proof}

\subsection{Total orders on indices}

Let $x$ be an alcove.  For an integer $m \in \Z$, let
\[
\Fix(x,m) = \{ i \in \{1,\ldots,n\} \mid \text{$x^{(1)}_i = m$ and $\sP(x)(i) = i$} \}.
\]
In other words, $\Fix(x,m)$ is the set of $\sP(x)$-fixed points whose corresponding coordinate in $x^{(1)}$ has value $m$. The proof of the following lemma is a straightforward consequence of this definition.

\begin{lem}\label{lem:po-extend}
Let $x$ be an alcove.
\begin{enumerate}
\item Two distinct elements $i,j \in \{1, \ldots, n \}$ are incomparable under $\trianglelefteq_x$ if and only if there is an $m \in \Z$ such that $i, j \in \Fix(x,m)$.
\item Let $i,j,k \in \{1, \ldots, n\}$.  Suppose that $i,j \in \Fix(x,m)$ but $k \notin \Fix(x,m)$, so that $i$ and $j$ are each comparable to $k$.  We have $i \trianglelefteq_x k$ (resp.~$i \trianglerighteq_x k$) if and only if $j \trianglelefteq_x k$ (resp.~$j \trianglerighteq_x k$).
\item We have\label{it:po-ex-fixbij}
\[
\Fix(\rot(x),m) = \cox^{-1}(\Fix(x,m)).
\]
\end{enumerate}
\end{lem}

Lemma~\ref{lem:po-extend} has the following consequence:
\begin{equation}
\begin{minipage}{4in}
Refining $\trianglelefteq_x$ to a total order on $\{1,\ldots, n\}$ is equivalent to choosing a total order on $\Fix(x,m)$ for each $m \in \Z$.
\end{minipage}
\end{equation}

Suppose now that we have a refinement of $\trianglelefteq_x$ to a total order $\blacktriangleleft_x$, and a refinement of $\trianglelefteq_{\rot(x)}$ to a total order $\blacktriangleleft_{\rot(x)}$.  We say that the pair
\[
\blacktriangleleft_x,\quad \blacktriangleleft_{\rot(x)}
\]
is \emph{compatible} if for each $m \in \Z$, the total order on $\Fix(x,m)$ induced by $\blacktriangleleft_{x}$ agrees with that on $\Fix(\rot(x),m)$ induced by $\blacktriangleleft_{\rot(x)}$ under the bijection of Lemma~\ref{lem:po-extend}\eqref{it:po-ex-fixbij}.  Obviously, given a total order $\blacktriangleleft_x$ refining $\trianglelefteq_x$, there is a unique compatible total order $\blacktriangleleft_{\rot(x)}$ refining $\trianglelefteq_{\rot(x)}$.

\begin{lem}\label{lem:to-compat}
Let $x$ be an alcove, and let
\[
\blacktriangleleft_x,\quad \blacktriangleleft_{\rot(x)}
\]
be a compatible pair of total orders that refine $\trianglelefteq_x$ and $\trianglelefteq_{\rot(x)}$, respectively.  Let $i,j \in \{1,\ldots,n\}$, and assume that $i,j \ne \sP(x)(1)$.  Then
\[
i \blacktriangleleft_x j
\qquad\text{if and only if}\qquad
\cox^{-1}(i) \blacktriangleleft_{\rot(x)} \cox^{-1}(j).
\]
Furthermore, if $\sP(x)(1) \blacktriangleleft_x j$, then $\cox^{-1}(\sP(x)(1)) \blacktriangleleft_{\rot(x)} \cox^{-1}(j)$.
\end{lem}
\begin{proof}
This follows from Proposition~\ref{prop:po-compare} and the definition of  ``compatible.''
\end{proof}

\begin{lem}\label{lem:to-seq}
Suppose we have a sequence of total orders
\[
\blacktriangleleft_x, \quad \blacktriangleleft_{\rot(x)}, \quad \ldots, \quad \blacktriangleleft_{\rot^{n-1}(x)}
\]
on $\{1,\ldots,n\}$, where $\blacktriangleleft_{\rot^j(x)}$ refines $\trianglelefteq_{\rot^j(x)}$, and where consecutive terms are compatible pairs.  Then the pair
\[
\blacktriangleleft_{\rot^{n-1}(x)},\quad \blacktriangleleft_x
\]
is also compatible.
\end{lem}
\begin{proof}
For $1 \le k \le n-1$, it follows from the definition of ``compatible pair'' that for $i, j \in \Fix(\rot^k(x),m)$, we have
\[
i \blacktriangleleft_{\rot^k(x)} j
\qquad\text{if and only if}\qquad
\cox^{k}(i) \blacktriangleleft_x \cox^{k}(j).
\]
Let $\blacktriangleleft'$ be the unique total order refining $\trianglelefteq_x = \trianglelefteq_{\rot^n(x)}$ such that
\[
\blacktriangleleft_{\rot^{n-1}(x)},\quad \blacktriangleleft'
\]
is a compatible pair.  Then, for $i, j, \in \Fix(x,m) = \Fix(\rot^n(x),m)$, we have
\begin{align*}
i \blacktriangleleft' j\quad &\Longleftrightarrow\quad \cox(i) \blacktriangleleft_{\rot^{n-1}(x)} \cox(j) \\
&\Longleftrightarrow\quad \cox^{n-1}(\cox(i)) \blacktriangleleft_{x} \cox^{n-1}(\cox(j)) \\
&\Longleftrightarrow\quad i \blacktriangleleft_{x} j.
\end{align*}
In other words, $\blacktriangleleft'$ coincides with $\blacktriangleleft_x$.
\end{proof}

\begin{lem}\label{lem:to-perm}
Let $x$ be an alcove.  Choose a total order $\blacktriangleleft_x$ refining $\trianglelefteq_x$, and then extend it to a sequence 
\[
\blacktriangleleft_x, \quad \blacktriangleleft_{\rot(x)}, \quad \ldots, \quad \blacktriangleleft_{\rot^{n-1}(x)}
\]
as in Lemma~\ref{lem:to-seq},  For $k \in \{1, \ldots, n\}$, let $\delta^{(k)} \in \fS_n$ be the permutation such that
\[
\delta^{(k)}(i) \le \delta^{(k)}(j)
\qquad\text{if and only if}\qquad
\cox^{1-k}(i) \blacktriangleleft_{\rot^{k-1}(x)} \cox^{1-k}(j).
\]
Then $\delta^{(k)}(x^{(k)} - \varpi_{k-1})$ is a dominant coweight.
\end{lem}
\begin{proof}
We begin by proving the lemma in the special case where $k = 1$.  We must show that $\delta^{(1)}(x^{(1)})$ is dominant, i.e., that
\[
x^{(1)}_{(\delta^{(1)})^{-1}(1)} \ge x^{(1)}_{(\delta^{(1)})^{-1}(2)} \ge \cdots \ge x^{(1)}_{(\delta^{(1)})^{-1}(n)}.
\]
Suppose one of these is false, say $x^{(1)}_{(\delta^{(1)})^{-1}(k)} < x^{(1)}_{(\delta^{(1)})^{-1}(k+1)}$.  Then $(\delta^{(1)})^{-1}(k+1) \trianglelefteq_x (\delta^{(1)})^{-1}(k)$, and hence $(\delta^{(1)})^{-1}(k+1) \blacktriangleleft_x (\delta^{(1)})^{-1}(k)$.  But this contradicts the definition of $\delta^{(1)}$, which implies that $(\delta^{(1)})^{-1}(k) \blacktriangleleft_x (\delta^{(1)})^{-1}(k+1)$.

Next, observe that
\[
\delta^{(k)}(\cox^{k-1}(i)) \le \delta^{(k)}(\cox^{k-1}(j))
\qquad\text{if and only if}\qquad
i \blacktriangleleft_{\rot^{k-1}(x)} j.
\]
Thus, the element $\delta^{(k)} \cox^{k-1} \in \fS_n$ defined in terms of $x$ coincides with the permutation $\delta^{(1)}$ defined in terms of $\rot^{k-1}(x)$.  By the previous paragraph, the coweight
\[
\delta^{(k)}(\cox^{k-1}(\rot^{k-1}(x)^{(1)}))
\]
is dominant.  Recall from Lemma~\ref{lem:rot-basic} that $\rot^{k-1}(x)^{(1)} = \cox^{1-k}(x^{(k)} - \varpi_{k-1})$.  We conclude that $\delta^{(k)}(x^{(k)} - \varpi_{k-1})$ is dominant.
\end{proof}

\begin{lem}\label{lem:delta-compare} 
Let $x$ be an alcove.  Choose a compatible sequence of total orders as in Lemma~\ref{lem:to-seq}, and let $\delta^{(1)}, \ldots, \delta^{(n)} \in \fS_n$ be as in Lemma~\ref{lem:to-perm}.  Let $t \in \{1,\ldots,n-1\}$.  For each $k \in \{1,\ldots, n\}$, there is an integer $b(t,k)$ such that 
\[
(\delta^{(k+1)})^{-1}(\varpi_t) = (\delta^{(k)})^{-1}(\varpi_t) - e_{b(t,k)} + e_{\sP(x)(k)},
\] where in case $k=n$, we set $\delta^{(n+1)}=\delta^{(1)}$. 
\end{lem}
\begin{proof}
List the integers $1, \ldots, n$ in order according to $\blacktriangleleft_{\rot^{k-1}(x)}$ and $\blacktriangleleft_{\rot^k(x)}$, as follows:
\begin{gather}
a_1 \blacktriangleleft_{\rot^{k-1}(x)} a_2 \blacktriangleleft_{\rot^{k-1}(x)} \cdots \blacktriangleleft_{\rot^{k-1}(x)} a_n, \label{eqn:dr-ai}\\
b_1 \blacktriangleleft_{\rot^{k}(x)} b_2 \blacktriangleleft_{\rot^{k}(x)} \cdots \blacktriangleleft_{\rot^{k}(x)} b_n.\label{eqn:dr-bi}
\end{gather}
Denote the smallest $t$ elements with respect to each of these orders as follows:
\[
S = \{a_1, a_2, \ldots, a_t \}, 
\qquad
S' = \{b_1, b_2, \ldots, b_t \}.
\]

Recall that
\[
\sP(\rot^{k-1}(x))(1) = \cox^{1-k}(\sP(x)(\cox^{k-1}(1))) = \cox^{1-k}(\sP(x)(k)).
\]
Applying to Lemma~\ref{lem:to-compat} to the alcove $\rot^{k-1}(x)$, we see that if $i, j \ne \cox^{1-k}(\sP(x)(k))$, we have $i \blacktriangleleft_{\rot^{k-1}(x)} j$ if and only if $\cox^{-1}(i) \blacktriangleleft_{\rot^k(x)} \cox^{-1}(j)$.  We also have that $\cox^{1-k}(\sP(x)(k)) \blacktriangleleft_{\rot^{k-1}(x)} j$ implies $\cox^{-k}(\sP(x)(k)) \blacktriangleleft_{\rot^k(x)} \cox^{-1}(j)$.  Now let $m$ and $m'$ be such that
\[
a_m = \cox^{1-k}(\sP(x)(k)),
\qquad
b_{m'} = \cox^{-1}(a_m) = \cox^{-k}(\sP(x)(k)).
\]
With this notation, we rephrase the conclusion of Lemma~\ref{lem:to-compat} as follows:
\begin{align}
i \blacktriangleleft_{\rot^{k-1}(x)} j &\Longleftrightarrow \cox^{-1}(i) \blacktriangleleft_{\rot^k(x)} \cox^{-1}(j) &&\text{if $i, j \ne a_m$,} \notag \\
a_m \blacktriangleleft_{\rot^{k-1}(x)} j &\Longrightarrow b_{m'} \blacktriangleleft_{\rot^k(x)} \cox^{-1}(j). \label{eqn:dr-toc}
\end{align}
In other words, if we omit $a_m$ from the sequence~\eqref{eqn:dr-ai} and then apply $\cox^{-1}$, we obtain the sequence in~\eqref{eqn:dr-bi} with $b_{m'}$ omitted.  Moreover,~\eqref{eqn:dr-toc} implies that
\[
m' \le m.
\]
Applying these observations to the smallest $t$ elements in~\eqref{eqn:dr-ai} and~\eqref{eqn:dr-bi}, we find
\begin{equation}\label{eqn:dr-scomp}
S' =
\begin{cases}
\cox^{-1}(S) & \text{if $t < m'$ or $m \le t$,} \\
(\cox^{-1}(S) \smallsetminus \{\cox^{-1}(a_t)\}) \cup \{b_{m'}\} & \text{if $m' \le t < m$.}
\end{cases}
\end{equation}

Now, from the  definition of $\delta^{(k)}$, we have
\begin{multline*}
\cox^{1-k}((\delta^{(k)})^{-1}(1)) \blacktriangleleft_{\rot^{k-1}(x)} 
\cox^{1-k}((\delta^{(k)})^{-1}(2)) \blacktriangleleft_{\rot^{k-1}(x)} \\
\cdots
\blacktriangleleft_{\rot^{k-1}(x)} \cox^{1-k}((\delta^{(k)})^{-1}(n)).
\end{multline*}
Comparing this to~\eqref{eqn:dr-ai}, we obtain
\[
a_i = \cox^{1-k}((\delta^{(k)})^{-1}(i) \qquad\text{for $1 \le i \le n$.}
\]
Similar considerations shows that
\[
b_i = \cox^{-k}((\delta^{(k+1)})^{-1}(i) \qquad\text{for $1 \le i \le n$.}
\]
We therefore have
\[
(\delta^{(k)})^{-1}(\varpi_t) = \sum_{i=1}^t e_{(\delta^{(k)})^{-1}(i)} = \sum_{j \in \cox^{k-1}(S)} e_j,
\qquad
(\delta^{(k+1)})^{-1}(\varpi_t) = \sum_{j \in \cox^k(S')} e_j.
\]
Combining these formulas with~\eqref{eqn:dr-scomp}, we find that
\[
(\delta^{(k+1)})^{-1}(\varpi_t) =
\begin{cases}
(\delta^{(k)})^{-1}(\varpi_t) & \text{if $t < m'$, or if $m \le t$}, \\
(\delta^{(k)})^{-1}(\varpi_t) - e_{\cox^{k-1}(a_t)} + e_{\cox^k(b_{m'})} & \text{if $m' \le t < m$.}
\end{cases}
\]
Recall that $\cox^k(b_{m'}) = \sP(x)(k)$.  Let
\[
b(t,k) =
\begin{cases}
\sP(x)(k) & \text{if $t < m'$ or $m \le t$,} \\
\cox^{k-1}(a_t) & \text{if $m' \le t < m$.}
\end{cases}
\]
We thus have $(\delta^{(k+1)})^{-1}(\varpi_t) = (\delta^{(k)})^{-1}(\varpi_t) - e_{b(t,k)} + e_{\sP(x)(k)}$ in all cases.
\end{proof}

\begin{thm}\label{thm:convolve-comb}
Let $\mu, \lambda \in \bX^+$ be dominant coweights such that $\lambda = \mu + \varpi_t$ for some $t\in\{1,\dots,n-1\}$.  For any $\lambda$-permissible alcove $x$, there exists a $\mu$-permissible alcove $y$ such that $x$ is in relative position $\varpi_t$ with respect to $y$.
\end{thm}

\begin{proof}
Let $\delta^{(1)}, \ldots, \delta^{(n)} \in \fS_n$ be as in Lemma~\ref{lem:to-perm}.  Define a sequence of coweights $y = (y^{(1)}, \ldots, y^{(n)})$ by
\[
y^{(k)} = x^{(k)} - (\delta^{(k)})^{-1}(\varpi_t).
\]
We have $|y^{(k)}| = |x^{(k)}| - t$, and since $x$ is an alcove, it follows that
\[
|y^{(2)}| = |y^{(1)}|+1,
\quad
|y^{(3)}| = |y^{(2)}|+1,
\quad
\ldots
\quad
|y^{(n)}| = |y^{(n-1)}|+1.
\]
Next, we will show that
\begin{equation}\label{eqn:cc-ineq}
y^{(1)} \le_\co y^{(2)} \le_\co \cdots \le_\co y^{(n)} \le_\co y^{(1)} + \varpi_n.
\end{equation}
For $1 \le k \le n-1$, by Lemma~\ref{lem:delta-compare}, we have
\begin{multline*}
y^{(k+1)} = x^{(k+1)} - (\delta^{(k+1)})^{-1}(\varpi_t) \\
= (x^{(k)} + e_{\sP(x)(k)}) - ( (\delta^{(k)})^{-1}(\varpi_t) - e_{b(t,k)} + e_{\sP(x)(k)}) = y^{(k)} + e_{b(t,k)},
\end{multline*}
and this establishes the first $n-1$ inequalities in~\eqref{eqn:cc-ineq}.  But in fact the same calculation is also valid when $k = n$, if we set $x^{(n+1)} = x^{(1)} + \varpi_n$ and $y^{(n+1)} = y^{(1)} + \varpi_n$.  We conclude that $y$ is an alcove.

For $1 \le k \le n$, we have $\dom(x^{(k)} - y^{(k)}) = \dom ((\delta^{(k)})^{-1}(\varpi_t)) = \varpi_t$, so $x$ is in relative position $\varpi_t$ with respect to $y$.  Finally, observe that
\[
\delta^{(k)}(y^{(k)} - \varpi_{k-1}) + \varpi_t = \delta^{(k)}(x^{(k)} - \varpi_{k-1}) = \dom(x^{(k)} - \varpi_{k-1}) \preceq \lambda.
\]
In particular, $\delta^{(k)}(y^{(k)} - \varpi_{k-1}) + \varpi_t$ is dominant.  Lemma~\ref{lem:dom-subtract} then tells us that
\[
\dom(y^{(k)} - \varpi_{k-1}) = \dom (\delta^{(k)}(y^{(k)} - \varpi_{k-1})) \preceq \mu.
\]
Thus, $y$ is $\mu$-permissible.
\end{proof}

\section{Convolution}
\label{sec:conv}

Let $X$ and $Y$ be two subsets of $\Gr$, and assume that $Y$ is stable under $\GL_n(\bO)$.  Their \emph{twisted product} is the set
\[
X \ttimes Y = \left\{ (L, L') \in \Gr \times \Gr \,\Big|\,
\begin{array}{@{}c@{}}
\text{$L \in X$, and for some (or any) $g \in \GL_n(\bK)$} \\
\text{such that $g \cdot L = \bO^n$, we have $g \cdot L' \in Y$}
\end{array}
\right\}.
\]
Let us check that this is well-defined: if $g, g' \in \GL_n(\bK)$ both satisfy $g \cdot L = g' \cdot L = \bL^0$, then $g'g^{-1}$ preserves $\bO^n$, which means that $g'g^{-1}  \in \GL_n(\bO)$.  Since $Y$ is $\GL_n(\bO)$-stable, we see that $g \cdot L' \in Y$ if and only if $g' \cdot L' \in Y$.  Let
\[
\mult: X \ttimes Y \to \Gr
\]
be the projection map onto the second factor: $\mult(L,L') = L'$.  

As an example, for $1 \le k \le n-1$, we have (cf.~Example~\ref{ex:locfund})
\begin{equation}\label{eqn:globtwist-fund}
X \ttimes \Gr_{\varpi_k} = 
\left\{ (L, L') \in \Gr \times \Gr \,\Big|\,
\begin{array}{@{}c@{}}
\text{$L \in X$, $tL \subset L' \subset L$,} \\ \text{and $\dim L'/tL = n-k$}
\end{array}
\right\}.
\end{equation}

There is an alternative description of $X \ttimes Y$ as a quotient of a certain space with a free $\GL_n(\bO)$-action.  This description shows that the \emph{first} projection map $X \ttimes Y \to X$ makes $X \times Y$ into a locally trivial fibration over $X$ with fibers isomorphic to $Y$.  In particular, In particular, if $X$ and $Y$ are irreducible subvarieties of $\Gr$, it follows that $X \ttimes Y$ is also irreducible.

\begin{lem}\label{lem:conv-image}
The image of $\mult: \overline{\Gr_\lambda} \ttimes \overline{\Gr_\mu} \to \Gr$ is $\overline{\Gr_{\lambda + \mu}}$.
\end{lem}
\begin{proof}[Proof Sketch]
Using~\eqref{eqn:gr-i-closure}, one can check that the largest dominant $\nu$ (with respect to $\preceq$) such that $\bL^\nu$ lies in this image is $\nu = \lambda + \mu$.  On the other hand, because $\mult: \overline{\Gr_\lambda} \ttimes \overline{\Gr_\mu} \to \Gr$ is proper and $\GL_n(\bO)$-equivariant, and because its domain is irreducible, its image must be a closed, irreducible, $\GL_n(\bO)$-invariant subvariety, and hence the closure of some $\GL_n(\bO)$-orbit.  The lemma follows.
\end{proof}

There is also a ``global'' version of the twisted product construction.  Defining it in full generality requires the introduction of a ``global group scheme'' that exhibits $I$ as a degeneration of $\GL_n(\bO)$; see~\cite[\S3.2]{zhu:ccpr}. In this paper, we will make do with an \emph{ad hoc} definition that is sufficient for our purposes.  For $\lambda, \mu \in \bX^+$, let
\[
X_\lambda \ttimes_\C X_\mu = \left\{ 
\begin{array}{r}
(y,L_1,\ldots,L_n, L'_1, \ldots, L'_n) \quad\\
\quad \in \bGr \times_\C \bGr 
\end{array}
\,\Big|\,
\begin{array}{c}
\text{for $1 \le i \le n$, we have} \\
(L_i,L_i') \in \overline{\Gr_\lambda} \ttimes \overline{\Gr_\mu}
\end{array}
\right\}.
\]
Let
\[
\bmult: X_\lambda \ttimes_\C X_\mu \to \bGr
\]
be the second projection map: $\bmult(y, L_1,\ldots, L_n, L'_1, \ldots, L'_n) = (y, L'_1, \ldots, L'_n)$.

Let $(X_\lambda \ttimes_\C X_\mu)_\spc$, resp.~$(X_\lambda \ttimes_\C X_\mu)_\gen$, denote the preimage of $0$, resp.~$\C^\times$, under the projection map $X_\lambda \ttimes_\C X_{\mu} \to \C$.  By Lemma~\ref{lem:bgr-gen-spc}, we have $(X_\lambda \ttimes_\C X_\mu)_\gen \cong \C^\times \times (\overline{\Gr_\lambda} \ttimes \overline{\Gr_\mu})$, and we can make the following identification:
\begin{equation}\label{eqn:conv-spc}
(X_\lambda \ttimes_\C X_{\mu})_\spc \subset \Fl \times \Fl.
\end{equation}

\begin{lem}\label{lem:globconv-irred}
Let $\lambda \in \bX^+$, and let $1 \le k \le n-1$.  If $X_\lambda$ is irreducible, then so is $X_\lambda \ttimes_\C X_{\varpi_k}$.
\end{lem}
\begin{proof}
Let $\cV$ be the variety
\[
\cV = \{ (L,v) \mid L \in \Gr,\ v \in L/tL \}.
\]
This is a vector bundle over $\Gr$ of rank $n$.  Next, let
\[
\bV = \{ (y, (L_1,v_1), (L_2,v_2), \ldots, (L_n,v_n)) \in \C \times \cV \times \cdots \times \cV \mid (y, L_1, \ldots, L_n) \in \bGr \}.
\]
This is a vector bundle over $\bGr$ of rank $n^2$.  Now, fix a point
\begin{equation}\label{eqn:bv-point}
(y, (L_1,v_1), (L_2,v_2), \ldots, (L_n,v_n)) \in \bV.
\end{equation}
Recall that $\theta_{i-1}(y)(L_i) \subset L_1$, so $\theta_{i-1}(y): \bK^n \to \bK^n$ determines a map $\theta_{i-1}(y): L_i/tL_i \to L_1/tL_1$.  We may therefore consider the sequence of vectors
\begin{equation}\label{eqn:bv-vecseq}
v_1,\quad  \theta_1(y)(v_2),\quad  \theta_2(y)(v_3),\quad  \ldots,\quad \theta_{n-1}(y)(v_n) \qquad\in L_1/tL_1.
\end{equation}
Let
\begin{multline*}
\widetilde{\bGr} = \{ (y, (L_1,v_1), (L_2,v_2), \ldots, (L_n,v_n)) \in \bV \mid \\
\text{$v_1,\  \theta_1(y)(v_2),\  \ldots,\  \theta_{n-1}(y)(v_n)$ are linearly independent} \}.
\end{multline*}
This is an open subset of $\bV$.  Moreover, its intersection with each fiber of $\bV \to \bGr$ is nonempty: if $y \ne 0$, then the maps $\theta_{i-1}(y): L_i/tL_i \overset{\sim}{\to} L_1/tL_1$ are isomorphisms, and any basis for $L_1/tL_1$ determines a linearly independent sequence as in~\eqref{eqn:bv-vecseq}. If $y = 0$, then $L_1 \supset \theta_1(0)L_2 \supset \cdots \supset \theta_{n-1}(0)(L_n) \supset tL_1$ is a lattice chain, and we can choose each $v_i$ so that $\theta_{i-1}(0)(v_i) \notin \theta_i(0)L_{i+1}$ to get a linearly independent sequence as in~\eqref{eqn:bv-vecseq}.

Finally, let
\[
\widetilde{X_\lambda} = X_\lambda \times_{\bGr} \widetilde{\bGr}.
\]
This is a nonempty open subset of the vector bundle $X_\lambda \times_{\bGr} \bV$ over $X_\lambda$.  

Given a point as in~\eqref{eqn:bv-point}, we have more generally that
\[
\begin{cases}
\theta_{i-1}(y)^{-1}\theta_{j-1}(y)(L_j) \subset L_i & \text{if $j \ge i$,} \\
\theta_{i-1}(y)^{-1}\theta_{j-1}(y)\theta_n(y)(L_j) \subset L_i & \text{if $j < i$.}
\end{cases}
\]
We can generalize~\eqref{eqn:bv-vecseq} and obtain a sequence of vectors
\begin{multline}\label{eqn:bv-vecseq2}
\hbox{\small$\theta_{i-1}(y)^{-1}\theta_n(y)(v_1), \ \theta_{i-1}(y)^{-1}\theta_1(y)\theta_n(y)(v_2), \ \ldots,\  \theta_{i-1}(y)^{-1}\theta_{i-2}(y)\theta_n(y)(v_{i-1}),$} \\
\hbox{\small$v_i,\  \theta_{i-1}(y)^{-1}\theta_i(y)(v_{i+1}), \ \ldots,\ \theta_{i-1}(y)^{-1}\theta_{n-1}(y)(v_n)$} \qquad \in L_i/tL_i
\end{multline}
for each $i$, $1 \le i \le n$.  If our point~\eqref{eqn:bv-point} actually lies in $\widetilde{\bGr}$, then the set in~\eqref{eqn:bv-vecseq2} is again linearly independent for each $i$.  In this case, we denote by
\[
g_i: \C^n \to L_i/tL_i
\]
the linear isomorphism that sends $e_1, \ldots, e_n \in \C^n$ to the basis vectors listed in~\eqref{eqn:bv-vecseq2}.  If $H \subset \C^n$ is a linear subspace, say of codimension $k$, then the subspace $g_i(H) \subset L_i/tL_i$ corresponds to some lattice sandwiched between $L_i$ and $tL_i$ that has valuation $\nu(L_i) + k$.  In a minor abuse of notation, we denote this lattice by
\[
g_i(H) + tL_i.
\] We may also write $g_i(L)$ for a lattice $L$ satisfying $t^n\bO^n\subset L\subset \bO^n$. In that case, $g_i(L) = g_i(L/t^n\bO^n)+tL_i$. 

We are now ready to define a morphism of varieties
\[
q: \widetilde{X_\lambda} \times_\C X_{\varpi_k} \to X_\lambda \ttimes_\C X_{\varpi_k},
\]
using the description from~\eqref{eqn:xfund-defn} for $X_{\varpi_k}$.  We set
\begin{multline}\label{eqn:qmap-defn}
q(y, (L_1,v_1),(L_2,v_2), \ldots, (L_n,v_n), H_1, \ldots, H_n) =\\
(y,L_1,L_2,\ldots, L_n, g_1(H_1) + tL_1, g_2(H_2) + tL_2, \ldots , g_n(L_n) + tL_n).
\end{multline}
It is immediate from~\eqref{eqn:globtwist-fund} that each pair $(L_i, g_i(H_i)+tL_i)$ lies in $\overline{\Gr_\lambda} \ttimes \Gr_{\varpi_k}$, so $q$ does indeed take values in $X_\lambda \ttimes_\C X_{\varpi_k}$.  We also claim that $q$ is surjective.  Indeed, given a point
\begin{equation}\label{eqn:qmap-surj}
(y,L_1, \ldots, L_n, L'_1, \ldots, L'_n) \in X_\lambda \ttimes_\C X_{\varpi_k},
\end{equation}
choose any point $(y, (L_1,v_1),\ldots,(L_n,v_n)) \in \widetilde{X_\lambda}$ lying over $(y, L_1, \ldots, L_n) \in X_\lambda$.  Then~\eqref{eqn:qmap-surj} is the image under $q$ of the point
\[
(y, (L_1,v_1),\ldots,(L_n,v_n), g_1^{-1}(L'_1/tL_1), g_2^{-1}(L'_2/tL_2), \ldots, g_n^{-1}(L'_n/tL_n)).
\]

We are now ready to complete the proof of the lemma.  Since $X_\lambda$ is assumed to be irreducible, so is the vector bundle $X_\lambda \times_{\bGr} \bV$ and its open subset $\widetilde{X_\lambda}$.  The variety $X_{\varpi_k}$ is also irreducible, by Theorem~\ref{thm:main-fund}.  Since $q$ is surjective, we conclude that $X_\lambda \ttimes_\C X_{\varpi_k}$ is irreducible.
\end{proof}

\begin{cor}\label{cor:globconv-image}
Let $\lambda \in \bX^+$, and let $1 \le k \le n-1$.  If $X_\lambda$ is irreducible, then the image of $\bmult: X_\lambda \ttimes_\C X_{\varpi_k} \to \bGr$ is $\overline{\bGr_{\lambda + \varpi_k}}$.
\end{cor}
\begin{proof}
By~\eqref{eqn:conv-spc} and Lemma~\ref{lem:conv-image}, the generic part of this image is $\C^\times \times \overline{\Gr_{\lambda+\varpi_k}}$.  On the other hand, $\bmult: X_\lambda \ttimes_\C X_{\varpi_k} \to \bGr$ is proper, and its domain is irreducible by Lemma~\ref{lem:globconv-irred}, so its image is closed and irreducible, and hence equal to $\overline{\bGr_{\lambda+\varpi_k}}$.
\end{proof}

\begin{ex}\label{ex:alcove-basis}
Let $x$ be an alcove, and consider the point $\ubL^x \in \Fl$.  Via Lemma~\ref{lem:bgr-gen-spc}, this corresponds to the point
\[
(0, \bL^{x^{(1)}}, \bL^{x^{(2)} - \varpi_1}, \ldots, \bL^{x^{(n)} - \varpi_{n-1}}) \in \bGr_\spc.
\]
Let
\[
v_i = \theta_{i-1}(0)^{-1} t^{x^{(1)}_{\sP(x)(i)}} e_{\sP(x)(i)} \qquad\text{for $1 \le i \le n$.}
\]
Then the vectors $v_1, \theta_1(0)(v_2), \ldots, \theta_{n-1}(0)(v_n)$ form an $\bO$-basis for $\bL^{x^{(1)}}$, and then yield a $\C$-basis for $\bL^{x^{(1)}}/t\bL^{x^{(1)}}$.  The corresponding map $g_1: \bO^n \to \bL^{x^{(1)}} \subset \bK^n$ is given by the matrix
\[
g_1 = \bt^{x^{(1)}} \dot\sP(x) = \dot \sP(x) \bt^{\sP(x)^{-1} \cdot x^{(1)}}.
\]
Similar considerations show that we have
\begin{multline}\label{eqn:alcove-gi}
g_i = \bt^{-\varpi_{i-1}}g_1 \bt^{\varpi_{i-1}} = \dot \sP(x) \bt^{\sP(x)^{-1}(x^{(1)} - \varpi_{i-1}) + \varpi_{i-1}} \\
= \dot\sP(x) \bt^{\sP(x)^{-1} \cdot (x^{(i)} - \varpi_{i-1})}: \bO^n \to \bL^{x^{(i)} - \varpi_{i-1}}.
\end{multline}
\end{ex}

\begin{lem}\label{lem:alcove-basis}
Let $\mu \in \bX^+$, let $1 \le k \le n-1$.  Let $y$ be a $\mu$-permissible alcove, and let $x$ be another alcove in relative position $\varpi_k$ with respect to $y$.  Then the pair $(\ubL^y, \ubL^x) \in \Fl \times \Fl$ lies in $(X_\mu \ttimes_\C X_{\varpi_k})_\spc$ under the identification~\eqref{eqn:conv-spc}.
\end{lem}
\begin{proof}
Consider the element $(\sP(y)^{-1}\sP(x), \sP(x)^{-1}(x^{(1)} - y^{(1)})) \in \Wext$.  Let
\[
z = (\sP(y)^{-1}\sP(x), \sP(x)^{-1}(x^{(1)} - y^{(1)})) \cdot \omega
\]
be the corresponding alcove.  Explicitly, the $i$-th term of $z$ is given by
\begin{multline*}
z^{(i)} = \sP(y)^{-1}\sP(x)\cdot( \sP(x)^{-1}(x^{(1)} - y^{(1)}) + \varpi_{i-1}) \\
= \sP(y)^{-1}( x^{(1)} + \sP(x)(\varpi_{i-1}) - y^{(1)} - \sP(y)(\varpi_{i-1})) + \varpi_{i-1} \\
= \sP(y)^{-1}(x^{(i)} - y^{(i)}) + \varpi_{i-1}.
\end{multline*}
Since $\dom(x^{(i)} - y^{(i)}) = \varpi_k$ by assumption, we see explicitly that $z$ is $\varpi_k$-permissible.  Thus, $\ubL^z$ is a point in $X_{\varpi_k} \cap \bGr_\spc$.

Next, consider the following point in $\widetilde{X_\mu} \times X_{\varpi_k}$ given by
\begin{multline}\label{eqn:convcalc-point}
\big(\quad (0,\  (\bL^{y^{(1)}}, t^{y^{(1)}_{\sP(y)(1)}} e_{\sP(y)(1)}),\ (\bL^{y^{(2)} - \varpi_1}, \theta_1(0)^{-1}  t^{y^{(1)}_{\sP(y)(2)}} e_{\sP(y)(2)}), \ldots,\\(\bL^{y^{(n)} - \varpi_{n-1}}, \theta_{n-1}(0)^{-1}  t^{y^{(1)}_{\sP(y)(n)}} e_{\sP(y)(n)})),\\ (0, \bL^{z^{(1)}}, \bL^{z^{(2)} - \varpi_1}, \ldots, \bL^{z^{(n)} - \varpi_{n-1}})\quad \big),
\end{multline}
where the component in $\widetilde{X_\mu}$ is given by the construction in Example~\ref{ex:alcove-basis} (with $x$ replaced by $y$).
Using formula~\eqref{eqn:alcove-gi}, we see that
\[
g_i(\bL^{z^{(i)} - \varpi_{i-1}}) = \dot \sP(y) \bt^{\sP(y)^{-1}(y^{(i)} - \varpi_{i-1})} \bL^{\sP(y)^{-1}(x^{(i)} - y^{(i)})} = \bL^{x^{(i)} - \varpi_{i-1}}.
\]
Thus, the image of~\eqref{eqn:convcalc-point} under the map $q$ from the proof of Lemma~\ref{lem:globconv-irred} is
\[
(0, \bL^{y^{(1)}}, \bL^{y{(2)} - \varpi_1}, \ldots, \bL^{y^{(n)} - \varpi_{n-1}}, \bL^{x^{(1)}}, \bL^{x^{(2)} - \varpi_1}, \ldots, \bL^{x^{(n)} - \varpi_{n-1}}),
\] 
which corresponds to $(\ubL^y, \ubL^x) \in \Fl \times \Fl$ under~\eqref{eqn:conv-spc}.
\end{proof}

We are finally ready to prove the main theorem of the paper.

\begin{proof}[Proof of Theorem~\ref{thm:main}]
In the special case where $\lambda$ is a fundamental coweight, we have already established the result in Theorem~\ref{thm:main-fund}.

For general $\lambda$, we proceed by induction on $\lambda_1 - \lambda_n$.  If $\lambda_1 - \lambda_n = 0$, then $\lambda$ is a constant coweight, and the result holds by Lemma~\ref{lem:adm-constant}.  Suppose now that $\lambda_1 - \lambda_n > 0$.  Then there is some $k$ with $1 \le k \le n-1$ such that $\lambda_k > \lambda_{k+1}$.  Let $\mu = \lambda - \varpi_k$.  The coweight $\mu$ is again dominant, and it satisfies $\mu_1 - \mu_n < \lambda_1 - \lambda_n$, so by induction, $X_\mu = \overline{\bGr_\mu}$.  In particular, $X_\mu$ is irreducible, so by Lemma~\ref{lem:globconv-irred}, so is $X_\mu \ttimes_\C X_{\varpi_k}$.

Let $x$ be a $\lambda$-permissible alcove.  By Theorem~\ref{thm:convolve-comb}, there exists a $\mu$-permissible alcove $y$ such that $x$ is in relative position $\varpi_k$ with respect to $y$.  By Lemma~\ref{lem:alcove-basis}, $(\ubL^y, \ubL^x)$ lies in $(X_\mu \ttimes_\C X_{\varpi_k})_\spc$, and then by Corollary~\ref{cor:globconv-image}, its image $\bmult(\ubL^y, \ubL^x) = \ubL^x$ lies in $\overline{\bGr_\lambda} \cap \bGr_\spc$.  We  have established~\eqref{eqn:main-claim}, so $X_\lambda = \bGr_\lambda$.
\end{proof}

\end{document}